\documentclass[11pt,a4paper]{article}
\usepackage{authblk}

\usepackage{graphicx}

\usepackage{amssymb}
\usepackage{amsmath}
\usepackage{amsfonts}
\usepackage{color}

\newtheorem{theorem}{Theorem}[section]
\newtheorem{proposition}[theorem]{Proposition}

\newenvironment{proof}{\removelastskip \noindent {\bf Proof~:} } { \hspace*{\fill} {\bf q.e.d.} \medskip \noindent}
\newtheorem{remark}{Remark}[section]
\newtheorem{corollary}[theorem]{Corollary}

\title{On the $N$-Extended Euler System I.\\  
Generalized Jacobi Elliptic Functions}
\author[]{S. Ferrer}
\author[]{F. Crespo}
\author[]{F. J. Molero}
\affil[]{{\small Dpto de Matem\'atica Aplicada, Universidad de Murcia, 30071 Espinardo,  Spain}}

\begin{document}
\maketitle

\begin{abstract}
We study the integrable system of first order differential equations $\omega_i(v)'=\alpha_i\,\prod_{j\neq i}\omega_j(v)$, $(1\!\leq i, j\leq \! N)$ as an initial value problem, with real coefficients $\alpha_i$ and initial conditions $\omega_i(0)$. The analysis is based on its quadratic first integrals. For each dimension $N$, the system defines a family of functions, generically hyperelliptic functions.  When $N=3$, this system  generalizes the classic Euler system for the reduced flow of the free rigid body problem, thus we call it $N$-extended Euler system ($N$-EES).  In this Part I the cases $N=4$ and $N=5$ are studied, generalizing Jacobi elliptic functions which are defined as a 3-EES. Taking into account the nested structure of the $N$-EES, we propose repa\-ra\-metrizations of the type ${\rm d}v^*=g(\omega_i)\,{\rm d}v$ that separate geometry from dynamic.  Some of those parametrizations turn out to be genera\-li\-zation of the {\sl Jacobi amplitude}.  In Part II we consi\-der geometric properties of the $N$-system and the numeric computation of the functions involved. It  will be published elsewhere.

{\bf keywords:} Integrable systems \and Generalized Euler system \and Jacobi and Weierstrass elliptic functions \and third Legendre elliptic integral

\end{abstract}

\section{Introduction}
\label{intro}
We are interested in the real functions $\omega_i(v)$ which are solutions of the integrable system of differential equations  
\begin{equation}\label{eq:EESn}
\frac{{\rm d}\omega_i}{{\rm d}v}=\alpha_i\,\prod_{j\neq i}\omega_j,\qquad (1\!\leq i ,j\leq \! N),
\end{equation}
with coefficients and initial conditions $\alpha_i, \omega_i(0)\!\in\!\mathbb{R}$.
Our study is based on the quadratic expressions 
 \begin{equation}\label{eq:integralescuadraticas}
C_{ij}(v) = \alpha_i\,\omega_j(v)^2 - \alpha_j\,\omega_i(v)^2
\end{equation}
which are integrals of the system (\ref{eq:EESn}). 
Initial conditions (IC) will be denoted $\omega^0\equiv \omega(0)=(\omega_1(0),\ldots,\omega_n(0))$. To simplify expressions we will use as notation $\omega_i\equiv\omega_i(v)$ and $\omega_i'\equiv{\rm d}\omega_i/{\rm d}v$. 

From the geometric point of view, the integrals (\ref{eq:integralescuadraticas}) tell us that the flow defined by (\ref{eq:EESn}) is the result of the intersection of quadrics in dimension $N$; more precisely, elliptic and hyperbolic cylinders. Thus, the $N$-EES family belongs to a larger family where the paraboloids are also included, as well as the degenerate cases defined by the hyperplanes. Its Poisson structure is defined by a determinant built on the gradients of the independent integrals, {\it i.e.} the Casimirs. When $N=3$ the classic mixed product  is precisely the determinant: one of the  integrals is the Casimir and the other the Hamiltonian; details will be given elsewhere \cite{Crespo2015}.

One of the features of the system (\ref{eq:EESn}) is that it allows, from a dynamical system point of view, dealing with a large family of functions in the real domain  in a unified way. It ranges from trigonometric functions (harmonic oscillator) to elliptic functions (pendulum and free rigid body), including also rational functions (for unbounded trajectories), etc. We will learn that different systems will allow us to introduce the same functions. For instance the hyperbolic functions may be introduced with $N=2$, but also appear when $N=3$ and two of the coefficients are equal). The interest of the study of the generic system $N> 4$ (the case $N=4$ is special, as we show below) lies in the fact that we face then hyperelliptic integrals and their inverses, a well established theory of special functions of complex variable made in XIX century which, nowadays, is in a revival in several branches of science, particularly in mechanics. But, although the theory is `at hand', nevertheless its application results a nontrivial task, because of the number of parameters involved in the definition of the functions, solutions of an IVP. 

\subsection{On Euler system, Jacobi functions and 3-EES}
\label{sec:Jacobi}
In this paper, our program is to generalize {\sl Jacobi elliptic functions}. Thus,  within the dynamical system point of view we have adopted, let us remember how all this started. The history of the $N$-EES begins with the well known Euler system of nonlinear differential equations in three dimensions \cite{MarsdenRatiu}, giving the reduced dynamics of the free rigid body problem (the dynamics of the angular momentum vector ${\bf \Pi}$ in the moving frame)
\begin{equation}\label{Euler}
\Pi_1'= \alpha_1\,\Pi_2\Pi_3, \,\,\Pi_2'= \alpha_2\,\Pi_1\Pi_3, \,
\Pi_3'= \alpha_3\,\Pi_1\Pi_2,
\end{equation}
such that $\sum\alpha_i=0$, where $\alpha_i$ are functions of the principal moments of inertia.

Associated with (\ref{Euler}), the second fundamental system, known as 
the {\sl Jacobi system}, is given by
\begin{equation}\label{Jacobi}
\omega_1'= \omega_2\,\omega_3,  \quad\omega_2'= -\omega_1\,\omega_3, \quad
\omega_3'= -m\,\omega_1\,\omega_2,
\end{equation}
with $\omega(0)=(0,1,1)$. The functions solution of (\ref{Jacobi}), denoted as $\omega_1\equiv {\rm sn}, \,\omega_2\equiv {\rm cn}$ and $\omega_3\equiv {\rm dn}$, are called {\sl Jacobi elliptic functions}. Then, the solution of (\ref{Euler}) are given by means of those functions, using the method of undetermined coefficients. For some readers could be useful to consult our paper \cite{CrespoFerrer} where we have studied  the {\sl extended Euler system} \begin{equation}\label{sistema3}
\omega_1'= \alpha_1\,\omega_2\omega_3,  \quad\omega_2'= \alpha_2\,\omega_1\omega_3, \quad
\omega_3'= \alpha_3\,\omega_1\omega_2,
\end{equation}
{\it i.e.} the (\ref{eq:EESn}) for $N=3$, considering generic values for coefficients $\alpha_i$ and initial conditions defining the system. 

Relying on the work of Tricomi \cite{Tricomi}, Hille \cite{Hille} and Meyer \cite{Meyer} dedicated to system (\ref{Jacobi}),  we have shown  in a straightforward manner how Jacobi and Weierstrass elliptic functions in the real domain are connected with this system \cite{CrespoFerrer}, although the tradition is to treat them separately because of the their intrinsic differences in the complex domain (see for instance Whittaker and Watson \cite{Whittaker} and Lawden \cite{Lawden}). Here we will apply the same approach to the system in $N$-dimensions. More precisely, we will present the generalization of both types of functions, where the $N$-Weierstrass function relates with the norm of the vector defined by the functions $\omega_i$.
\subsection{Integrals, functions and regularization }
\label{sec:regularization}
Moreover, as an alternative to confront directly with hyperelliptic functions, we propose {\sl to experiment} with repa\-ra\-metrizations starting from low dimensions. More precisely, we extend the regularization ${\rm d}v^*=\omega_3{\rm d}v$,  already studied for the case $N=3$ by Molero {\it et al.} \cite{Molero}. This way of proceeding seems to be an open line of work. The fact that elliptic and hyperelliptic functions are `naturally' introduced within the context of complex functions may explain why we have not found references. It is due to the consideration of those functions in a dynamical systems context, in the real domain, that the regularization enters on the scene. More precisely we focus on `regularizations' of the type ${\rm d}v^*=g(\omega_i){\rm d}v$, a technique well known in classical fields such as Celestial Mechanics (where they are used for studies ranging from collisions to efficient numerical integration schemes). We will see that the new variable is a generalization of the Jacobi amplitude.  This procedure, based on the symmetry of the system, alleviates the manipulation of the hyperelliptic functions involved, which are relegated to only one quadrature (the {\sl regularization equation}), separating it from the geometry (it is part of our research, knowing more on how generic this procedure is). 

This research has two parts. Part I, which makes the content of this paper,  works in detail the cases $N=4,5$. The key aspect associated with this case is that for each IVP we deal with two or three parameters. In Section~\ref{sec:generic} we briefly refers to the equilibria as well as particular solutions such as the rectilinear. After that we fix the dimension considering the case 4-EES. In Section~\ref{sec:N4ratios} we present a basic feature related to the ratios of the functions. In Section~\ref{sec:Mahler} we focus in a biparametric system, which we dubbed as {\sl Mahler system}.  In Section~\ref{sec:regularizacion} we apply to our system the regularization technique. We identify that the new variable is a `generalized amplitude'.
In Sect. \ref{sec:additionformulas} we provide with the addition formulas associated to the Mahler system. Using them we propose extending the work of Bulirsch and Fukushima, we introduce some formulas related to the numerical evaluation of a 4-EES.
In Section~\ref{sec:N5} we approach the system for $N=5$, focusing in one of the particular cases, showing its connection with the previous dimension. Finally, as an application, we briefly consider in Section~\ref{sec:FRB} the free rigid body formulated in Andoyer variables

For the benefit of the reader we include two Appendices which contain properties of $\theta_i$ and elliptic Jacobi functions. There is a Part II, devoted to generic features of (\ref{eq:EESn}) from the geometric point of view, and to the numeric evaluation of the Mahler system, following the steps of Bulirsch and Fukushima. This  will be published elsewhere.

We ought to close the Introduction pointing out that this paper does not contain 
a complete analysis of the relative role of the parameters involved in the defined functions. Some transformations related to the range of those parameters are required, similar to the well known transformations for the elliptic modulus of the Jacobi functions. That analysis is still in progress.

\section{Some basic features of $N$-EES}
\label{sec:generic}
We have mentioned in the Introduction that our interest in this paper focuses on the study of some systems (\ref{eq:EESn}) of low dimension. Nevertheless, as in any dimension common features are present, it is worth to briefly refer to some of them.

\subsection{On particular solutions: equilibria and straight lines through the origin}  
Before we start our analysis of the IVP,  a first question is to identify the equilibria of the system (\ref{eq:EESn}). Denoting $P=(p_1,p_2,\ldots,p_n)$ an equilibrium point, we easily check that the system has the following set of equilibria:
\begin{itemize}
\item Origin $P=0\in \mathbb{R}^n$,
\item For $n\geq 3$, the points: $P_i=(0,\ldots, p_i,\ldots,0)$, \quad $1\leq  {i}\leq n$,  functions of the initial conditions.
\item For $n\geq 4$, planes $\Pi_{i_1,i_2}=(0,\ldots, p_{i_1},\ldots,p_{i_2},\ldots  0)$, \quad $1\leq  {i_1}<{i_2}\leq n$, functions of the initial conditions.
\item For $n\geq 5$, the hyperplanes 
$$\Pi_{i_1,i_2,\ldots,i_{n-2}}=(0,\ldots, p_{i_1},\ldots,p_{i_2},\ldots, p_{i_{n-2}},\ldots 0),$$  $1\leq  {i_1}<{i_2}< {i_{n-2}}\leq n$.
\end{itemize}
Thus, associated to these equilibria hyperplanes, we have the study 
of their invariant manifolds and their connections, generalizing the heteroclinic trajectories in three dimensions. This is out of the scope of the present paper. 
\medskip\par\noindent
{\sl Straight-lines through the origin}. Meanwhile in the generic study of the quadratures associated with our system (see Sect.~\ref{sec:quadrature}) an assumption is commonly made, namely, the roots of the polynomials involved are different, when considering an IVP we may be under a scenario where we have multiple roots. This is precisely the case with {\sl straight-lines through the origin}. Then, instead of requiring the use of special functions,  the solutions are expressed by means of {\sl elementary functions}, different for each dimension. 

\subsection{Reduction to quadratures: Generalized Weierstrass function}
\label{sec:quadrature}
Taking into account the integrals (\ref{eq:integralescuadraticas}), and proceeding like in the classic case $N=3$, we may reduce the system to a fundamental differential equation in two forms. The first one, after choosing one of te functions, say $\omega_i$, it leads to the differential equation
\begin{equation}
\big(\frac{{\rm d}\omega_i}{{\rm d} v}\big)^2 = \alpha_i^{3-N}\,\big[\prod_{j\neq i}^N(\alpha_j\omega_i^2+C_i^j)\big]. 
\end{equation}
or, by separation, the corresponding quadrature
\begin{equation}
\alpha_i^{(3-N)/2}\,v=\int\frac{{\rm d}\omega_i}{[\prod_{j\neq i}^N(\alpha_j\omega_i^2+C_i^j)]^{1/2}}. 
\end{equation}
As an alternative, if we introduce the {\sl square of the norm} 
\begin{equation}\label{squarenorm}
\Omega_N(v)\equiv\omega(v)^2=\sum_{i=1}^N\omega_i(v)^2,
\end{equation}
after some straightforward computations we obtain
\begin{equation}\label{weierstrassquadrature0}
\Big(\frac{{\rm d}\Omega_N}{{\rm d} v}\Big)^2 = 4\,\prod_{i=1}^N (\Omega_N-b_i), \quad  \sum_{i=1}^N b_i=0,
\end{equation}
a differential equation whose solution $\Omega_N(v)$ may be seen as the generalized Weierstrass function $\wp(v)$. Following either way we confront generically hyperelliptic integrals.

\subsection{On the normalized $N$-EES}
\label{sec:normalized}
Associated to a generic $N$-EES (\ref{eq:EESn}), {\it i.e.} assuming that $\sum\alpha_i\neq 0$, we consider the {\sl square norm} function (\ref{squarenorm}) that satisfies
\begin{equation}
\frac{{\rm d}\omega}{{\rm d}v}=(\sum_{i=1}^N\alpha_i)\frac{1}{\omega}\prod_{i=1}^N \omega_i.
\end{equation}
Thus, introducing the functions $$\tilde\omega_i=\frac{\omega_i}{\omega},$$
we have
\begin{equation}
\frac{{\rm d}\phantom{-}}{{\rm d}v}\Big(\frac{\omega_i}{\omega}\Big)=[\alpha_i\omega^2-\left(\sum_{i=1}^N\alpha_i\right)\omega_i^2]\frac{1}{\omega^3}\prod_{j\neq i}^N \omega_j.
\end{equation}
which may be written also as
\begin{equation}
\frac{{\rm d}\tilde\omega_i}{{\rm d}v}=c_i\,\prod_{j\neq i}^N \tilde\omega_j\,\omega^{N-4},
\end{equation}
where the coefficients 
\begin{equation}
c_i=\alpha_i\omega^2-(\sum\alpha_i)\omega_i^2
\end{equation}
are integrals of the flow, whose values are determined for each IVP by the initial conditions. 
In other words, carrying out the reparametrization $v\rightarrow v^*$ given by
\begin{equation}\label{newquadrature}
{\rm d}v^*=\omega^{N-4}\,{\rm d}v,
\end{equation}
associated to (\ref{eq:EESn})  we have the {\sl normalized system}
\begin{equation}\label{sistemaNnormal}
\frac{{\rm d}\tilde\omega_i}{{\rm d}v^*}=c_i\prod_{j\neq i}^N \tilde\omega_j,
\end{equation}
with initial conditions
\begin{equation}
\tilde\omega_i(0)=\omega_i(0)/\omega(0), \quad \omega(0)^2=\sum\omega_i(0)^2,
\end{equation}
\emph{i.e.} the flow (\ref{sistemaNnormal}) lives in $\mathbb{S}^{N-1}$ and, like  the differential system satisfied by the angular momentum in 3-D, we have $\sum c_i=0$. Note that to deal with the system (\ref{sistemaNnormal}) versus (\ref{eq:EESn}) will bring advantages, at least from the numerical point of view.

With (\ref{sistemaNnormal}) integrated  we have $\tilde\omega_i=\tilde\omega_i(v^*)$. Then, we still have to implement the quadrature associated to the regularization (\ref{newquadrature}) in order to recover the relation with the original variable. For instance, considering the first integral $c_1$  we obtain
\begin{equation}
{\rm d}v= \omega^{4-N}\,{\rm d}v^*= \Big(\frac{c_1-(\sum\alpha_i)\tilde\omega_1(v^*)^2}{\alpha_1}\Big)^{\frac{4-N}{2}}\,{\rm d}v^*
\end{equation}
whose quadrature gives the parametrization relation,  solved generically by numeric methods.
Note that, the case $N=4$ is special, because we do not need to do regularization.

Moreover, we will not pursue here with the study of the normalized system (\ref{sistemaNnormal}). For details on this analysis we refer to \cite{Crespo2015}.
 
Let us close this Section  pointing out another  basic feature of this system; we refer to it as the {\sl scaling factor}.  If the functions $\omega_i(v), \, (i=1,\ldots N)$ is a set of  solutions, then taking a constant $c$, the functions $u_i(v)=c\,\omega_i(c^{N-2}v)$ satisfy the same system with  the corresponding IC given by $u_i(0)=c\,\omega_i(0)$. We will make use of this property along the paper.
\section{The case $N=4$. Relying on Jacobi elliptic functions?}
\label{sec:N4ratios}
We focus now on the 4-EES case. For each IVP, with some abuse of notation, we refer to the functions solutions generically with $\omega_i$. Later, referring to some specific systems, we will introduce new notations.

At this point, perhaps some readers would like to know the original motivation of our interest in 4-EES case. The reason is connected with an observation about the classical way in which the study of the rigid body dynamics is developed, based on Jacobi elliptic functions. Meanwhile those functions depend on one parameter (elliptic modulus), and appear naturally tied to problems like the pendulum or the measure of an arc of ellipse,  when we apply them to the rigid body problem, we need to consider a second parameter (the {\sl characteristic}, a function of the principal moments of inertia). In other words, the first and third Legendre elliptic integrals are involved. Since Jacobi, the way to proceed has been: (i) to introduce complementary functions $Z$ and $\Theta$; (ii) to make use of the addition formulas of elliptic functions, dealing with the second parameter as an amplitude, etc. Here we search for an alternative to such approach considering a generalization of Jacobi elliptic functions with two parameters.

Thus, we start with the 4-EES
\begin{equation}\label{sistema40}
\begin{array}{l}
\omega_1'= \alpha_1\,\omega_2\,\omega_3\,\omega_4, \\[1ex]
\omega_2'= \alpha_2\,\omega_1\,\omega_3\,\omega_4,\\[1ex]
\omega_3'= \alpha_3\,\omega_1\,\omega_2\,\omega_4,\\[1ex]
\omega_4'= \alpha_4\,\omega_1\,\omega_2\,\omega_3,
\end{array}
\end{equation}
with given initial conditions $\omega^0$, and the corresponding six quadratic first integrals (\ref{eq:integralescuadraticas}), of which three are independent (Fig.~\ref{fig:GenericSolution} shows a graph of the solution of the system (\ref{sistema40})). Although by scaling and a change of variables we could get rid of two of the coefficients $\alpha_i$, for our purpose it is convenient here to maintain all of them.
\begin{figure}[h!]
\centering
{\label{fig:SeccionesEenergia}\includegraphics[scale=1.3]{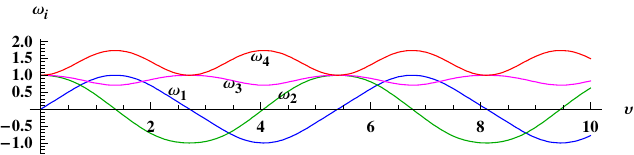}}
\vspace{-0.4cm}
\caption{\small Graphical solution of the previous system (\ref{sistema40}) for $\alpha_1=1;\,\alpha_2=-1;\,\alpha_3=2,\,\alpha_4=-0.5$.}
\label{fig:GenericSolution}
\end{figure}

To our surprise, the only reference we have found so far to (\ref{sistema40}) is E. Hille \cite{Hille}, where the case $N=4$ is considered in Chapter 2 (exercises 7, 8 and 9) under the suggestion of K. Mahler. More precisely he considers the IVP  $\omega(0)=(0,1,1,1)$ and coefficients $\alpha_i=(1,-1,-\alpha^2,-\beta^2)$, with both coefficients less than one. He says ``the solutions are hyperelliptic functions of genus 2'', a statement on which we disagree. Finally he mentions ``the example can be generalized in an obvious manner.''

Thus, our plan is: (i) to study (\ref{sistema40}) as an extension of the case $N=3$ where the Jacobi elliptic functions were defined. Note that represent a drastic reduction in the number of parameters (coefficients and IC) to discuss; (ii) To introduce again regularizations. In order to approach both aspects, apart from its own interest, we think the case $N=4$ is critical in the search for methodologies to follow when dealing with systems of higher dimension, {\it i.e.} in the reign of  hyperelliptic integrals.

\subsection{Nested structure and integration by Jacobi elliptic functions}

Extending what we know for the case $N=3$, a basic feature of the $N$-EES is its relation with the system verified by the ratios. Referring to that we say the 4-EES has a `nested structure', and we call it the `Glashier Ratios Property'. Moreover the case $N=4$ asks for a particular study devoted to it. As we will see, for other dimensions a regularization is needed.

\begin{proposition} {\bf (Glashier Ratios Property)}
\label{propo:glashier4} Given the functions $\omega_i(v)$ verifying a 4-EES, then the functions $\omega_i(v)/\omega_j(v)$ defined by their ratios, $(i,j,k,l)\in {\rm Per}(1,2,3,4)$ satisfy a 3-EES given by 
\begin{equation}\label{glashier4}
\begin{array}{l}
\displaystyle{\frac{{\rm d}\phantom{-}}{{\rm d}v}\Big(\frac{\omega_i}{\omega_l}\Big)=C_i^l\,\frac{\omega_j}{\omega_l}\frac{\omega_k}{\omega_l}},\\[1.5ex]
\displaystyle{\frac{{\rm d}\phantom{-}}{{\rm d}v}\Big(\frac{\omega_j}{\omega_l}\Big)=C_j^l\,\frac{\omega_i}{\omega_l}\frac{\omega_k}{\omega_l}},\\[1.5ex]
\displaystyle{\frac{{\rm d}\phantom{-}}{{\rm d}v}\Big(\frac{\omega_k}{\omega_l}\Big)=C_k^l\,\frac{\omega_i}{\omega_l}\frac{\omega_j}{\omega_l}},
\end{array}
\end{equation}
with initial conditions $\omega_i(0)/\omega_l(0)$ and coefficients given by
the integrals $C_i^l=\alpha_i\omega_l^2-\alpha_l\omega_i^2$.
\end{proposition}
\begin{proof} 
It is straightforward  making use of the definition of the 4-EES 
\end{proof} 
\begin{remark}
From the previous  Proposition \ref{propo:glashier4} readers familiar with the expressions of Jacobi elliptic functions,  and their computation by means of Jacobi theta functions $\theta_i(x)$, may wonder what the relation between those functions and the $\omega_i(v)$ might be. We have gathered some of those systems in an Appendix. In fact the reader will find in Lawden (Chp 1) a number of properties of $\theta_i$ functions which are also satisfied by the $\omega_i$. Perhaps, the simple fact that $\theta_1'(0)=\theta_2(0)\theta_3(0)\theta_4(0)$ is satisfied for the 4-EES when we take $\alpha_1=1$, is one of the most surprising. We will come back to this below.
\end{remark}
\begin{remark} Note that there is the possibility to take a slight different version of the ratios, namely to work with $u_j^i=c_j^i\,\omega_i/\omega_j$, with coefficients $c_j^i$ still to be determined, in order to simplify some expressions, adjust constants in applications, etc. We do not follow this alternative in this paper.
\end{remark}
\begin{proposition} 
\label{propo:sistema40}
For suitable IC the 4-EES (\ref{sistema40}) has as solution the bounded functions  $\omega_i(v)\equiv \omega_i(v;\alpha_i,\omega_i(0))$ given by
\begin{eqnarray}
&&\omega_1(v)=\tilde C_1^4\frac{{\rm sn}(a v|m_1)}{\sqrt{1-n_1\,{\rm sn}^2(a v|m_1)}},\label{omega41}\\
&&\omega_2(v)=\tilde C_2^4\frac{{\rm cn}(a v|m_1)}{\sqrt{1-n_1\,{\rm sn}^2(a v|m_1)}},\label{omega42}\\
&&\omega_3(v)=\tilde C_3^4\frac{{\rm dn}(a v|m_1)}{\sqrt{1-n_1\,{\rm sn}^2(a v|m_1)}},\label{omega43}\\
&&\omega_4(v)=\tilde C_4^4\frac{1}{\sqrt{1-n_1\,{\rm sn}^2(a v|m_1)}},\label{omega44}
\end{eqnarray}
where ${\rm sn}(a v|m_1)$, etc are the Jacobi elliptic functions, and the constants $\tilde C_i^4$, $a$, $m_1$ and $n_1$  are functions of $\alpha_i$ and $\omega_i(0)$.
\end{proposition}
\begin{proof} 
Let us assume IC $\omega^0=(\omega_1^0,\ldots, \omega_4^0)$ such that $\omega_j\neq 0$ in its domain of definition. According to the previous Proposition, we consider the ratios and the reciprocals $1/\omega_j$, that we denote
\begin{equation}\label{ratios4}
u_i^j=\frac{\omega_i}{\omega_j}, \quad  i\neq j, \qquad  u_j^j=\frac{1}{\omega_j},
\end{equation}
in the domain where $\omega_j$ is defined.  Without loss of generality we assume we refer to the case $j=4$, with IC such that $\omega_4>0$. Moreover, we still simplify a bit more the notation writing  $u_i^4=u_i$.

Then, according to Proposition \ref{propo:glashier4} it results for the functions  $u_i$, $i=1,2,3$ we have the following system
\begin{equation}\label{sistema430}
\begin{array}{l}
u_1'= C_1^4\,u_2\,u_3, \\[1ex]
u_2'= C_2^4\,u_3\,u_1,\\[1ex]
u_3'= C_3^4\,u_1\,u_2,
\end{array}
\end{equation}
with IC $u_i(0)=u_i^0=\omega_i^0/\omega_j^0$. Moreover, from the first integral 
\begin{equation}
\alpha_1\omega_4^2-\alpha_4\omega_1^2=C_1^4
\end{equation}
we may write
\begin{equation}\label{u4u1}
u_4^2=\frac{1}{C_1^4}(\alpha_1-\alpha_4 u_1^2).
\end{equation}
Because the functions  $u_i$, $i=1,2,3$ satisfy (\ref{sistema430}), they belong to the set of functions defined by the `Jacobi elliptic functions' ${\rm sn}, {\rm cn}, {\rm dn}$ and their ratios. Then, following Crespo and Ferrer \cite{CrespoFerrer},  we know our  system corresponds to one of the four possible cases (Glashier systems), depending on the sign of the integrals. Here, to continue our reasoning on the system (\ref{sistema40}), we focus on the case where the sign of  $C_1^4$ is different of  $C_2^4$ and $C_3^4$ (the other cases are treated likewise). This means that  $u_i$, $i=1,2,3$ are of the form, say
\begin{equation}\label{indeterminados}
\begin{array}{l}
u_1(v)=\delta_1\,{\rm sn}(a v,m_1), \\
u_2(v)=\delta_2\,{\rm cn}(a v,m_1),\\
u_3(v)=\delta_3\,{\rm dn}(a v,m_1).
\end{array}
\end{equation}
Proceeding by the method of undetermined coefficients, replacing (\ref{indeterminados}) in (\ref{sistema430}) we identify that the constants $\delta_i, a$ y $m_1$ satisfy a system of algebraic equations whose solution is\begin{eqnarray*}
&&\delta_2=u_2^0, \quad \delta_3=u_3^0,\quad\delta_1=\sqrt{-\alpha_1/\alpha_2}\delta_2, \\ 
&&a=\alpha_1\delta_2\delta_3/\delta_1, \quad m_1= \alpha_3\delta_2^2/(\alpha_2\delta_3^2)
\end{eqnarray*}
 (for details see for instance Lawden \cite{Lawden}, p. 132).  

Summarizing, according to (\ref{ratios4}) and (\ref{sistema430}) we have $\omega_i=u_i/u_4$, where $u_i$ (i=1,2,3) are the Jacobi elliptic functions and $u_4$ is given by (\ref{u4u1}). From those expressions, we obtain the functions (\ref{omega41})-(\ref{omega44}),  where  
\begin{equation}
\tilde C_4^4=\sqrt{C_1^4/\alpha_1},\quad  \tilde C_i^4=\delta_i/\tilde C_4^4,\quad  n_1=\alpha_4\delta_1^2/\alpha_1
\end{equation}
and, as stated in the Proposition, initial conditions still have to be chosen such that $n_1<1$.
\end{proof}
Before we continue it is convenient to formulate the previous Proposition in a `complementary form', where we make more transparent the role played by coefficients and initial conditions.
\begin{proposition} 
\label{pro:Jacobisimilar}
The functions $\omega_i(v)$, $i=1,\ldots 4$ , given by
\begin{equation}\label{soluciones}
\begin{array}{l}
\displaystyle{\omega_1(v)=\frac{\omega_2(0)\,\omega_3(0)\,\omega_4(0)}{a}\,
\frac{{\rm sn}(av|m_1)}{\sqrt{1+n_1\,{\rm sn}^2 (av| m_1)}}},\\
\displaystyle{\omega_2(v) =  \omega_2(0)\,\frac{{\rm cn}(av|m_1)}{\sqrt{1+n_1\,{\rm sn}^2(av|m_1)}}}, \\
\displaystyle{ \omega_3(v)=\omega_3(0)\,\frac{{\rm dn}(av|m_1)}{\sqrt{1+n_1\,{\rm sn}^2(av|m_1)}}},\\
\displaystyle{ \omega_4(v)= \omega_4(0)\,\frac{1}{\sqrt{1+n_1\,{\rm sn}^2(av|m_1)}}}.
\end{array}
\end{equation}
satisfy a differential system of the type (\ref{sistema40}) given by
\begin{equation}\label{Jacobisimilar0}
\begin{array}{l}
\displaystyle{\omega_1'= \phantom{-\,}\omega_2\,\omega_3\,\omega_4}, \\
\displaystyle{\omega_2'= - (1+n_1)\frac{a^2}{\omega_3^2(0)\omega_4^2(0)}\,\omega_1\,\omega_3\,\omega_4},\\
\displaystyle{\omega_3'= -(m_1+n_1)\frac{a^2}{\omega_2^2(0)\,\omega_4^2(0)}\,\omega_1\,\omega_2\,\omega_4},\\
\displaystyle{\omega_4'= -n_1\frac{a^2}{\omega_2^2(0)\omega_3^2(0)}\,\omega_1\,\omega_2\,\omega_3},
\end{array}
\end{equation}
with  $\omega= (0,\omega_2(0), \omega_3(0),\omega_4(0))$ as initial conditions
\end{proposition}
\begin{proof} 
It is a straightforward exercise by computing deri\-va\-ti\-ves.
\end{proof}

\begin{remark} 
In particular, choosing $\omega_i(0)=1$ $(i=2,3,4)$ and $a=1$, join with $n_1=n$ and $m_1=m-n$ in Proposition \ref{pro:Jacobisimilar}, we have the Jacobi elliptic functions
\begin{equation*}
{\rm sn}(v)=\frac{\omega_1(v)}{\omega_4(v)},\quad {\rm cn}(v)=\frac{\omega_2(v)}{\omega_4(v)},\quad{\rm dn}(v)=\frac{\omega_3(v)}{\omega_4(v)}
\end{equation*}
with elliptic modulus $m_1=m-n$, where $\omega_i(v;m,n)$ satisfy the system
\begin{equation}\label{ratiosJacobi}
\begin{array}{l}
\displaystyle{\omega_1'= \phantom{-\,}\omega_2\,\omega_3\,\omega_4}, \\
\displaystyle{\omega_2'= -(1+n)\,\omega_1\,\omega_3\,\omega_4},\\
\displaystyle{\omega_3'= -m\,\omega_1\,\omega_2\,\omega_4},\\
\displaystyle{\omega_4'= -n\,\omega_1\,\omega_2\,\omega_3},
\end{array}
\end{equation}
with integrals
\begin{equation}
\begin{array}{l}
\displaystyle{\omega_2^2+(1+n)\,\omega_1^2=1}, \\[1.1ex]
\displaystyle{\omega_3^2+m\,\omega_1^2=1},\\[1.1ex]
\displaystyle{\omega_4^2+n\,\omega_1^2=1}.
\end{array}
\end{equation}
If $0<n<m<1$, we have $-1/\sqrt{1+n}\leq \omega_1 \leq 1/\sqrt{1+n}$, $-1\leq \omega_2 \leq 1$, $\sqrt{1-m/(1+n)}\leq \omega_3 \leq 1$ and $\sqrt{1-n/(1+n)}\leq \omega_4 \leq 1$.

More details on the system (\ref{ratiosJacobi}) will not be given in the rest of this paper.
\end{remark}
\section{Studying two 4-EES systems}
Looking for the generalization of Jacobi elliptic functions, we now focus  on two cases of (\ref{sistema40}): 
\begin{itemize}
\item One-parameter ($\theta_i$ similar) family in Sec.~\ref{sec:thetas} and;
\item Two-parameter family (Mahler system) in Sec.~\ref{sec:Mahler}.
\end{itemize}
It is worth noting that the first two equations in both systems (see (\ref{Jacobisimilar1}) and (\ref{sistema4mn})) are equal, with the consequence that one of the integrals is $\omega_1^2+\omega_2^2=1$, which is not the case for the previous system (\ref{ratiosJacobi}).

In relation with both, before we continue, a comment on notation is due. In what follows,  it is convenient to redefine some of the constants which appear in the previous expressions. More precisely, in Sec.~\ref{sec:thetas} we write $m_1\equiv k^2$, and we will find that $a$ and $n_1$ are functions of $k$. Likewise,
in Sec.~\ref{sec:Mahler} we fix all initial conditions and coefficients  except two of them, denoted by $-m$ and $-n$. 
\subsection{One-parameter $\omega_i(v)$ functions, `similar' to Jacobi $\theta_i$ functions }
\label{sec:thetas}
We look here for functions $\omega_i$, solutions of our differential system (\ref{sistema40}), similar to Jacobi $\theta_i$ functions. What we mean by that should be made more precise: (i) coefficients and  IC of the 4-EES have to be dependent only of one parameter: $\alpha_i=\alpha_i(k)$, $\omega_i^0=\omega_i^0(k)$; (ii) Moreover those functions  $\omega_i(v;k)$ ought to be found imposing that they verify properties defining $\theta_i$ de Jacobi.

Such search does not appear straightforward because, we remember, $\theta_i$ functions are defined as 1-para\-meter Fourier series solving the heat equation. Our way of proceeding will be to take into account those properties of $\theta_i$ which could be imposed on the differential system: both the ratios and the identities satisfied by $\theta_i(0)$ are essential for us. 

\begin{proposition} {\rm\bf ($\omega_i$: `similar Jacobi $\theta_i$ functions')}
 Choosing initial conditions as functions of the elliptic modulus 
\begin{equation}\label{values1}
\omega_1(0)=0,\,\omega_2(0)=\sqrt{a\,k}, \, \omega_3(0)=\sqrt{a},\, \omega_4(0)=\sqrt{a\,k'}
\end{equation}
join with
\begin{equation}\label{values2}
a=\frac{2K}{\pi}, \qquad n_1=k'-1, \qquad  m_1=k^2
\end{equation}
where $k'=\sqrt{1-k^2}$, then we may write
\begin{equation}
\begin{array}{l}
v_1(\omega_3^2(0)z) =\displaystyle{\frac{\omega_3(0)}{\omega_2(0)}\, \frac{\omega_1(z)}{\omega_4(z)}},\\[2ex]
v_2(\omega_3^2(0)z) =\displaystyle{\frac{\omega_4(0)}{\omega_2(0)}\, \frac{\omega_2(z)}{\omega_4(z)}},\\[2ex]
v_3(\omega_3^2(0)z) =\displaystyle{\frac{\omega_4(0)}{\omega_3(0)}\, \frac{\omega_3(z)}{\omega_4(z)}}
\end{array}
\end{equation}
in other words, we express the Jacobi elliptic functions as ratios of the $\omega_i(v)$, in a similar way as Jacobi gave them with respect to the $\theta_i$ functions.
\end{proposition}
\begin{proof}.- It is a straightforward exercise replacing the previous values (\ref{values1}) and (\ref{values2}) in Proposition \ref{pro:Jacobisimilar}. The result is that the functions are
\begin{equation}\label{funcionesthetasimilar}
\begin{array}{l}
\displaystyle{\omega_1(z,k) =\sqrt{a\,kk'}\,\frac{{\rm sn}(u)}{\sqrt{1-(1-k')\,{\rm sn}^2(u)}}},\\
\displaystyle{\omega_2(z,k) =\sqrt{a\,k}\,\frac{{\rm cn}(u)}{\sqrt{1-(1-k')\,{\rm sn}^2(u)}}},\\
\displaystyle{\omega_3(z,k) =\sqrt{a}\,\frac{{\rm dn}(u)}{\sqrt{1-(1-k')\,{\rm sn}^2(u)}}},\\
\displaystyle{\omega_4(z,k) =\sqrt{a\,k'}\,\frac{1}{\sqrt{1-(1-k')\,{\rm sn}^2(u)}}},
\end{array}
\end{equation}
join with $u=az$. 

Thus the system (\ref{Jacobisimilar0}) given by
\begin{equation}\label{Jacobisimilar1}
\begin{array}{l}
\displaystyle{\omega_1'= \omega_2\,\omega_3\,\omega_4,} \\ [1.2ex]
\displaystyle{\omega_2'= -\omega_3\,\omega_4\,\omega_1},\\[1.2ex]
\displaystyle{\omega_3'= -\frac{1-k'}{k}\,\omega_4\,\omega_1\,\omega_2},\\ [1.2ex]
\displaystyle{\omega_4'= \frac{1-k'}{k}\,\omega_1\,\omega_2\,\omega_3,}
\end{array}
\end{equation}
with initial conditions (\ref{values1}), is the IVP we were looking for. Fig.~\ref{fig:ThetaSimilarm095} shows an example of a graph of this set of functions.
\end{proof}
\begin{figure}[h!]
\centering
{\label{fig:SeccionesEenergia}\includegraphics[scale=1.3]{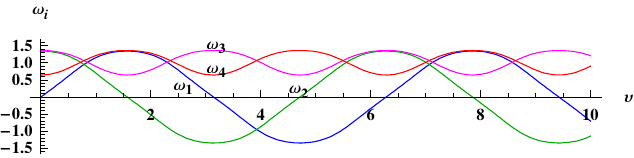}}
\vspace{-0.4cm}
\caption{\small Graph of the $\theta_i$-similar for $m=0.95$.}
\label{fig:ThetaSimilarm095}
\end{figure}

It is an exercise to check that the functions (\ref{funcionesthetasimilar}) verify  identical relations to the linear combinations  satisfied by the square of Jacobi $\theta_i$ functions (see Lawden, formulae (1.4.49)--(1.4.52), p. 11).

%
\subsection{Mahler system. A biparametric 4-EES:}
\label{sec:Mahler}
As a second distinguished 4-EES we consider now a `biparametric' case we call {\sl Mahler system}. It is an IVP  which defines the functions $\omega_i(v;m,n)$, solutions of (\ref{sistema40}) depending on two parameters,  such that 
\begin{itemize}
\item coefficients $\alpha=(1,-1,-m,-n)$ 
\item initial conditions $\omega^0=(0,1,1,1)$.
\end{itemize}
When $n=0$ then $\omega_i(v)$ are the Jacobi elliptic functions and $\omega_4(v)\equiv 1$.

Note that this represents some abuse of notation, because $n$ has already been used to denote the last component of an $N$-dimension system. Nevertheless, we think by the context it will become clear when is a coefficient: $n\in \mathbb{R}$, although in some occasions  $n$ might be used as a counter (ordinal number: $n\in \mathbb{N}$).

\begin{proposition} {\rm\bf (Mahler system)}\\
\label{propo:sistema4mn} 
The 4-EES given by  
\begin{equation}\label{sistema4mn}
\begin{array}{l}
\omega_1'= \,\omega_2\,\omega_3\,\omega_4, \\
\omega_2'= -\,\omega_1\,\omega_3\,\omega_4,\\
\omega_3'= -m\,\omega_1\,\omega_2\,\omega_4,\\
\omega_4'= -n\,\,\omega_1\,\omega_2\,\omega_3,
\end{array}
\end{equation}
where $n< m<1$, with IC $\omega(0)=(0,1,1,1)$, has  the functions
\begin{equation}\label{sistema4mnsolution}
\begin{array}{l}
\displaystyle{\omega_1=A\,\frac{{\rm sn}(av|m_1)}{\sqrt{1-n_1\,{\rm sn}^2(av|m_1)}}},\\
\displaystyle{\omega_2=\frac{{\rm cn}(av|m_1)}{\sqrt{1-n_1\,{\rm sn}^2(av|m_1)}}},\\
\displaystyle{\omega_3=\frac{{\rm dn}(av|m_1)}{\sqrt{1-n_1\,{\rm sn}^2(av|m_1)}}},\\
\displaystyle{\omega_4=\frac{1}{\sqrt{1-n_1\,{\rm sn}^2(av|m_1)}}},
\end{array}
\end{equation}
as solution, with values $a,A,m_1,n_1$  given by
\begin{equation}\label{valores}
\begin{array}{l}
a=\sqrt{1-n}, \quad \qquad A= 1/\sqrt{1-n}, \\[1.2ex]
\displaystyle{ n_1= \frac{n}{n-1}}, \quad\qquad \displaystyle{m_1= \frac{n-m}{n-1}}.
\end{array}
\end{equation}

\end{proposition}
\begin{proof}
Let us consider the system (\ref{sistema4mn}) as an IVP
with $\omega(0)= (0, \omega_2(0),\omega_3(0),\omega_4(0))$, $(\omega_i(0)\neq 0, \, i=2,3,4)$ dependent of two parameters $(m,n)$. It admits as solution the functions  
\begin{eqnarray*}
&&\tilde\omega_1(v) =A\,\frac{{\rm sn}(av|m_1)}{\sqrt{1-n_1\,{\rm sn}^2(av|m_1)}},\\
&&\tilde\omega_2(v) =\omega_2(0)\,\frac{{\rm cn}(av|m_1)}{\sqrt{1-n_1\,{\rm sn}^2(av|m_1)}},\\
&&\tilde\omega_3(v) =\omega_3(0)\,\frac{{\rm dn}(av|m_1)}{\sqrt{1-n_1\,{\rm sn}^2(av|m_1)}},\\
&&\tilde\omega_4(v) =\omega_4(0)\,\frac{1}{\sqrt{1-n_1\,{\rm sn}^2(av|m_1)}},
\end{eqnarray*}
where $a,A,m_1,n_1$ are given by 
\begin{equation}\label{parametersIC}
\begin{array}{l}
\displaystyle{a=\omega_3(0)\sqrt{\omega_4^2(0)-n\,\omega_2^2(0)}}, \\
\displaystyle{ A=\frac{\omega_2(0)\,\omega_4(0)}{\sqrt{\omega_4^2(0)-n\,\omega_2^2(0)}}}, \\
\displaystyle{ n_1=\frac{n\,\omega_2^2(0)}{n\,\omega_2^2(0)-\omega_4^2(0)}}, \\ 
\displaystyle{ m_1=\frac{\omega_2^2(0)\,(n\,\omega_3^2(0) - m\,\omega_4^2(0))}
          {\omega_3^2(0)\,(n\,\omega_2^2(0) - \omega_4^2(0))}}
\end{array}
\end{equation}
and the derivatives at the origin satisfy
\begin{equation}
\tilde\omega_1(0)'= \,\omega_2(0)\,\omega_3(0)\,\omega_4(0),\quad \tilde\omega_i(0)'=0, 
\end{equation}
 where $i=2,3,4$. Then, choosing as IC the quantities $\omega(0)=(0,1,1,1)$ and replacing them in (\ref{parametersIC}), we readily obtain the values (\ref{valores}) for those parameters.
\end{proof} 
\begin{remark}
In particular the  case $n=0$ leads to: $a=1, \, A= 1, \, n_1= 0$ y  $m_1= m,$ {\it i.e.}, the Jacobi elliptic functions. We have another special case when $m=0$. As we have assumed  $n<m$, in this case $n<0$ and the diferencial system (\ref{sistema4mn}) corresponds again to a Jacobi system, but now with negative parameter (there is a transformation to reduce it to the {\sl normal case}, see Appendix B, Sect. \ref{sec:Appendices}). For more on particular cases see   
Section~\ref{sec:particularcases}. We leave for the reader to work out the other particular cases defined by special values of the pair $(m,n)$. 
\end{remark}


%
\section{Regularization and `generalized amplitudes' for the Mahler system} 
\label{sec:regularizacion}
We have just solved the system  $N=4$ in the standard way: making use of known functions (Jacobi elliptic functions). In what follows we are going to proceed making use of the   {\sl regularization}. To do that, we start remembering in  Section \ref{sec:omega3} the recent proposal of the authors for $N=3$ (see Molero {\it et al.} \cite{Molero}),  which is intrinsically connected with the Jacobi {\sl amplitude}. After that we develop the same approach for the $N=4$ case. That proposal entails to study, at least, two possible regularizations $v\rightarrow v^*$ given by
\begin{itemize}
\item ${\rm d}v^*/{\rm d}v=\omega_4$,
\item ${\rm d}v^*/{\rm d}v=\omega_3\,\omega_4$,
\end{itemize}
which we gather in Sections \ref{sec:omega4} and \ref{sec:omega34}. Let us proceed one by one. But, before, we remember in Sect. \ref{sec:omega3} how this has been done for the 3-EES.
\subsection{Preliminaries: 3-EES and regularization}
\label{sec:omega3}
Let us consider the 3-EES (\ref{sistema3}) with initial conditions $\omega^0\equiv \omega(0)=(\omega_1(0),\omega_2(0),\omega_3(0))$, whose values we choose below. This system has the integrals
\begin{equation}\label{integrales3}
\alpha_1\omega_2^2-\alpha_2\,\omega_1^2=C_1^2, \qquad 
\alpha_1\omega_3^2-\alpha_3\,\omega_1^2=C_1^3.
\end{equation} 
Let us assume  $\alpha_i$ and IC such that $\omega_3(v)> 0$. Then, making use of the parametrization 
\begin{equation}\label{regularizacion}
\frac{{\rm d}v^*}{{\rm d}v}=\omega_3,
\end{equation}
the system (\ref{sistema3}) reduces to 
\begin{equation}\label{sistema2}
\frac{{\rm d}\omega_1}{{\rm d}v^*}= \alpha_1\,\omega_2, \qquad
\frac{{\rm d}\omega_2}{{\rm d}v^*}= \alpha_2\,\omega_1,
\end{equation}
join with the quadrature defined by (\ref{regularizacion}). Choosing the coefficients  $\alpha_1=1$, $\alpha_2=-1$ and IC $(\omega_1(0),\omega_2(0))=(0,1)$,  the system (\ref{sistema2}) defines the trigonometric (circular) functions:  
\begin{equation}\label{circulares}
\sin(v^*), \qquad \cos(v^*).
\end{equation}
(with other conditions, by a change of variables we may reduce it to this case)
Then, keeping in mind (\ref{integrales3}), the regularization (\ref{regularizacion}) takes the form
\begin{equation}\label{regularizacion1}
\frac{{\rm d}v^*}{{\rm d}v}=\sqrt{C_1^3+\alpha_3\,\omega_1^2}.
\end{equation}
Motivated by the dynamical system defining the {\sl simple pendulum}\footnote{This lead us to an interpretation of the regularization: $v\equiv t$ and $v^*\equiv \phi$, in other words `time' and `angle'. Angle in the 1-2 plane; arc through the integral $\omega_1^2+\omega_2^2=1$, a circle projection of the integral which is a cylinder.}, it is chosen $\omega_3(0)=1$ join with $\alpha_3=-k^2$, where $k^2<1$. Thus, replacing in (\ref{regularizacion1}) we have
\begin{equation}
{\rm d}v = \frac{{\rm d}v^*}{\sqrt{1-k^2\sin^2v^*}},
\end{equation}
whose quadrature and inversion leads us to the Jacobi  ``${\rm am}$'' function:
\begin{equation}
v^*= {\rm am}(v,k).
\end{equation}
Finally, replacing in (\ref{circulares}) we have the Jacobi functions 
\begin{equation}\label{elipticasJacobi}
\begin{array}{l}
\sin(v^*(v))=\sin({\rm am}(v,k)), \\
\cos(v^*(v))=\cos({\rm am}(v,k)),
\end{array}
\end{equation}
which today, following Gudermann,  are denoted in the form
\begin{equation*}\label{elipticasJacobi1}
\begin{array}{l}
{\rm sn}(v ;\, k)\equiv\sin({\rm am}(v,k)), \\
{\rm cn}(v ;\, k)\equiv\cos({\rm am}(v,k)).
\end{array}
\end{equation*}
Completing our set of functions $\omega_3$ is given by 
\begin{equation}
\omega_3(v)\equiv {\rm dn}(v ;\, k)=\sqrt{1-k^2{\rm sn}^2(v ;\, k)}.
\end{equation}
Summarizing, using the previous notation, the integrals 
 (\ref{integrales3}) lead us to the well known expressions relating these functions\begin{equation}
{\rm sn}^2+{\rm cn}^2 =1, \qquad {\rm dn}^2+ k^2{\rm sn}^2 = 1.
\end{equation}
Finally, replacing in  (\ref{sistema3}) we write what some authors refer as  ``{\sl derivation rules}'' of Jacobi functions:
\begin{equation}\label{derivadas}
{\rm sn}'= {\rm cn}\,{\rm dn}, \quad {\rm cn}'= -{\rm sn}\,{\rm dn},\quad
{\rm dn}'= -k^2{\rm sn}\,{\rm cn}.
\end{equation}
\subsection{The  ${\rm d}v^*/{\rm d}v=\omega_4$ regularization.}
\label{sec:omega4}
Proceeding as in the previous Section we treat now the case  $N=4$ by means of the regularization
\begin{equation}\label{regularizacion4}
\frac{{\rm d}v^*}{{\rm d}v}=\omega_4.
\end{equation}
\begin{remark}
Remember the comment above in relation with notation; although there is some abuse using again $v^*$ for denoting the new independent parameter, from the context we distinguish it from the one studied in the previous Section. 
\end{remark}
As a consequence the system (\ref{sistema40}) is reduced to 
\begin{equation*}
\frac{{\rm d}\omega_1}{{\rm d}v^*}= \alpha_1\,\omega_2\,\omega_3, \quad
\frac{{\rm d}\omega_2}{{\rm d}v^*}= \alpha_2\,\omega_1\,\omega_3,\quad
\frac{{\rm d}\omega_3}{{\rm d}v^*}= \alpha_3\,\omega_1\,\omega_2,
\end{equation*}
and $\omega_4(v^*)$ which will be obtained using one of the integrals, after we have solved the previous system.

We focus on the case $\alpha_1=1$, $\alpha_2=-1$ and $\alpha_3=-m$ because, as we have said before, we plan to generalize Jacobi elliptic functions. Thus, we have
\begin{equation}
\omega_1={\rm sn}(v^*;m_1), \, \omega_2={\rm cn}(v^*;m_1),\, \omega_3={\rm dn}(v^*;m_1)
\end{equation}
and for the differential relation \label{regularizacion4} using the integral $n\,\omega_1^2 + \omega_4^2=C_1^4$ and the initial conditions, we may write
\begin{equation}\label{cambioeliptico0}
v=\int\frac{{\rm d}v^*}{\sqrt{1 - n_1\,{\rm sn}^2(v^*;m_1)}}.
\end{equation}
%
\subsection{$N=4$. The regularization ${\rm d}v^*/{\rm d}v=\omega_3\,\omega_4$.}
\label{sec:omega34}
Proceeding the same way as for $N=3$, we treat now the case $N=4$ by means of the regularization
\begin{equation}\label{regularizacion34}
\frac{{\rm d}v^*}{{\rm d}v}=\omega_3\,\omega_4.
\end{equation}
As a consequence the system  (\ref{sistema40}) reduces to 
\begin{equation}\label{generic2}
\frac{{\rm d}\omega_1}{{\rm d}v^*}= \alpha_1\,\omega_2, \quad
\frac{{\rm d}\omega_2}{{\rm d}v^*}= \alpha_2\,\omega_1,
\end{equation}
and two quadratures associated to $\omega_3$ y $\omega_4$. In fact, they are not needed
because the integrals gives us  
$$\omega_i^2=C_1^i-\alpha_i\omega_1^2,\quad (i=3,4).$$ 
Note that $C_1^i$ are constants which depend on the initial conditions.

Without loss of generality we will assume our system is made of bounded functions. Then, by a change of variables, our system  (\ref{generic2}) reduces to $\alpha_1=1,\, \alpha_2=-1$, thus it results 
\begin{equation}\label{trig4}
\omega_1(v^*)=\sin v^*, \qquad \omega_2(v^*)=\cos v^*,
\end{equation}
Considering the previous integrals we may write (\ref{regularizacion34}) as follows
\begin{equation}
{\rm d}v = \frac{{\rm d}v^*}{\sqrt{\prod_{i=3}^4 (C_1^i-\alpha_i\sin^2 v^*)}}
\end{equation}
or in a slightly different form
\begin{equation}\label{cuadratura2}
\lambda{\rm d}v = \frac{{\rm d}v^*}{\sqrt{(1-\beta_1\sin^2 v^*)(1-\beta_2\sin^2 v^*)}}
\end{equation}
where $\beta_i$ and $\lambda$ are functions of $C_1^i$ and $\alpha_i$.  

In what follows, with the Mahler system in mind as the basic 4-EES, it is convenient
to take the associated notation:
$$\beta_1 \equiv n,\qquad \beta_2 \equiv m,\qquad \lambda\equiv 1.$$
In other words, the differential relation (\ref{cuadratura2}) reads
\begin{equation}\label{cuadraturadiferencial34}
{\rm d}v = \frac{{\rm d}v^*}{\sqrt{(1-n\sin^2 v^*)(1-m\sin^2 v^*)}}
\end{equation}
The quadrature takes the form
\begin{equation}
v=G(v^*,n,m)=\int_0^{v^*}\!\!\!\frac{d\vartheta}{\sqrt{(1-n\sin^2\vartheta)(1-m\sin^2\vartheta})}, 
\end{equation}
Thus, we define the period as the two-parameters function
\begin{equation}
G(\pi/2,n,m)=\int_0^{\pi/2}\frac{d\vartheta}{\sqrt{(1-n\sin^2\vartheta)(1-m\sin^2\vartheta})}
\end{equation}
Thus, when $(n,m)=(0,0)$, we have $G(0,0)=\pi/2$, and when $(n,m)=(1,1)$, we have $G(1,1)=\infty$.

When $(m,n)$ are small, if we carry out the Taylor expansion of the integrand, after the evaluation of the quadratures, $G(n,m)$ may be approximated in the form
\begin{eqnarray*}
&& G(n,m)=\\    
&&\hspace{0.2cm}\frac{\pi}{2}\Big[1 +\frac{m}{4}  + \frac{9 m^2}{64} + \frac{25 m^3}{256} + \frac{1225 m^4}{16384} \\
&&\hspace{0.7cm}+\frac{n}{4}\Big(1 +\frac{3m}{8}  + \frac{15 m^2}{64} + \frac{175 m^3}{1024} + \frac{2205 m^4}{16384}\Big) \\
&&\hspace{0.7cm}+\frac{9n^2}{64}\Big(1 +\frac{5m}{12}  + \frac{35 m^2}{128} + \frac{105 m^3}{512} + \frac{2695 m^4}{16384}\Big) \\
&&\hspace{0.7cm}+ \frac{25n^3}{256} \Big(1 +  \frac{7 m }{16} +\frac{189 m^2 }{640} + \frac{231 m^3 }{1024} + \frac{3003 m^4 }{16384}\Big) \\
 &&\hspace{0.7cm}+\frac{1225n^4}{16384}\Big(1+ \frac{9 m}{20} + \frac{99 m^2}{320}+ \frac{429 m^3}{1792} + \frac{6435 m^4}{32768}\Big)\Big]\\
 &&\hspace{0.7cm}+\rm{h.o.t.}
\end{eqnarray*}
although the previous expression may be written in different form making more explicit its symmetric character with respect to  $m$ and $n$.

Now  we define the {\sl generalized amplitud} {\rm amg}
as the inverse function
\begin{equation}
v^*= {\rm amg}(v; n,m).
\end{equation}
Thus, considering the expressions  (\ref{trig4}), we have
\begin{equation}
\sin\, v^*= \sin\,{\rm amg}(v,n,m)\equiv {\rm sng}(v,n,m)
\end{equation}
and
\begin{equation}
\cos\, v^*= \cos\,{\rm amg}(v,n,m)\equiv {\rm cng}(v,n,m)
\end{equation}
\par\noindent
$\bullet$ There is an alternative way of proceeding. If we consider the change of variable 
$\sin \vartheta=x$ it allows to follow the steps of Jacobi for the case $N=3$.
Then, the differential relation (\ref{cuadraturadiferencial34}) takes the form
\begin{equation}\label{cuadratura22}
{\rm d}v = \frac{{\rm d}x}{ \sqrt{(1-x^2)(1-n\,x^2) (1-m\,x^2) }}
\end{equation}
or, inverting the expression 
\begin{equation}\label{ed1}
 \frac{{\rm d}x}{{\rm d}v}=\sqrt{(1-x^2)(1-n\,x^2)(1-m\,x^2)}.
\end{equation}
In other words, we define the function ${\rm sng}$
\begin{equation}\label{sng}
x=x(v;n,m)={\rm sng}(v;n,m)
\end{equation}
as the two-parameters function (whose range is made more precise below), solution of the differential equation 
\begin{equation}
 \Big(\frac{{\rm d}x}{{\rm d}v}\Big)^2 = (1-x^2)(1-n\,x^2)(1-m\,x^2).
\end{equation}
In this paper we will restrict to a range $ n\leq m\leq 1$. 
\medskip\par\noindent
$\bullet$ Then, associated with ${\rm sng}$ we propose the following functions
\begin{eqnarray}
&&{\rm cng}(v;n,m)=\pm\sqrt{1-{\rm sng}^2(v;n,m)},\label{cng}\\[1.2ex]
&&{\rm dng}(v;n,m)=\sqrt{1-m\,{\rm sng}^2(v;n,m)},\label{dng}\\[1.2ex]
&&\,{\rm fng}(v;n,m)=\sqrt{1-n\,{\rm sng}^2(v;n,m)}.\label{fng}
\end{eqnarray}
To simplify notation we will write ${\rm sng}(v;n,m)\equiv {\rm sng}$, etc. Examples of the graph of these new functions can be seen in Figs.~\ref{fig:Mahlern01m08} and \ref{fig:Mahlernmenos2m05}.
\begin{figure}[h!]
\centering
{\label{fig:SeccionesEenergia}\includegraphics[scale=1.3]{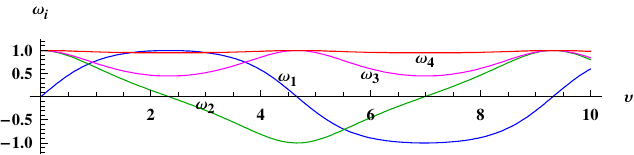}}
\vspace{-0.4cm}
\caption{\small Mahler $n=0.1, m=0.8$. Falta otra con valor m\'as extremo de $n=m$}
\label{fig:Mahlern01m08}
\end{figure}
\begin{figure}[h!]
\centering
{\label{fig:SeccionesEenergia}\includegraphics[scale=1.3]{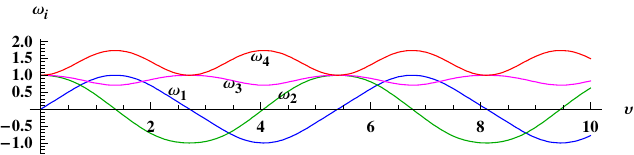}}
\vspace{-0.4cm}
\caption{\small Mahler $n=-2, m=0.5$. Falta otra con valor m\'as extremo de $n=m$}
\label{fig:Mahlernmenos2m05}
\end{figure}

Due to the process we have followed, we immediately check that these functions ${\rm sng}$, etc verify the following IVP
\begin{equation}\label{sistemasng}
\begin{array}{l}
\displaystyle{\frac{{\rm d}\,{\rm sng}}{{\rm d}v}= {\rm cng}\,{\rm dng}\,{\rm fng}}, \\[1.5ex] 
\displaystyle{\frac{{\rm d}\,{\rm cng}}{{\rm d}v}= -{\rm sng}\,{\rm dng}\,{\rm fng}},\\[1.5ex]
\displaystyle{\frac{{\rm d}\,{\rm dng}}{{\rm d}v}= -m\,{\rm sng}\,{\rm cng}\,{\rm fng}},\\[1.5ex]
\displaystyle{\frac{{\rm d}\,{\rm fng}}{{\rm d}v}= -n\,{\rm sng}\,{\rm cng}\,{\rm dng}},
\end{array}
\end{equation}
with initial conditions $(0,1,1,1)$. The integrals, as we have mentioned before, lead to the following expressions
\begin{equation}
{\rm cng}^2+{\rm sng}^2=1, \quad {\rm dng}^2+m\,{\rm sng}^2=1, \quad 
{\rm fng}^2+n\,{\rm sng}^2=1.
\end{equation}
 Thus, from the functions solution of the Mahler system, the Jacobi functions   are given by
\begin{equation}\label{JacobiMahler}
\begin{array}{l}
\displaystyle{{\rm sn}(av;m_1)=\frac{1}{A}\,\frac{{\rm sng}(v;m,n)}{{\rm fng}(v;m,n)}}, \\[2ex]
\displaystyle{{\rm cn}(av;m_1)= \frac{{\rm cng}(v;m,n)}{{\rm fng}(v;m,n)},} \\[2ex]
\displaystyle{{\rm dn}(av;m_1)=\frac{{\rm dng}(v;m,n)}{{\rm fng}(v;m,n)}, }
\end{array}
\end{equation}
\noindent
$\bullet$ {\sl Taylor expansions of ${\rm sng},\, {\rm cng},\, {\rm dng}$ and ${\rm fng}$ near the origin. }
\medskip\par\noindent
As a direct application of the definition of those functions by the differential system (\ref{sistemasng}), we may easily compute to any order the Taylor expansion of the previous functions:
\begin{equation}\label{algoritmo2}
\begin{array}{l}
\displaystyle{{\rm sng}(v)= v - \frac{1 + m + n}{6}\,v^3}\\[1.5ex]
\displaystyle{\hspace{0.5cm}+\frac{1 + 14(m + n + m n)+ m^2+ n^2}{120}\,v^5 +\ldots}\\[1.5ex]
\displaystyle{{\rm cng}(v)= 1 - \frac{1}{2}\,v^2+\frac{1 + 4 m + 4 n}{24}\,v^4}\\[1.5ex]
\displaystyle{\hspace{0.5cm}-\frac{1 + 44 (m +n)+ 16 m^2 + 104 m n + 16n^2}{720}\,v^6 +\ldots}\\[1.5ex]
\displaystyle{{\rm dng}(v)= 1 - \frac{m}{2}\,v^2+\frac{m(4 + m + 4 n)}{24}\,v^4}\\
\displaystyle{\hspace{0.5cm}-\frac{m(16 + 44 m + m^2 + 104 n + 44 m n + 16 n^2)}{720}\,v^6+\ldots}\\[1.5ex]
\displaystyle{{\rm fng}(v)= 1 - \frac{n}{2}\,v^2+\frac{n(4 + n + 4 m  )}{24}\,v^4}\\[1.5ex]
\displaystyle{\hspace{0.5cm}-\frac{n(16 + 44 n + n^2+ 104 m + 44 m n + 16 m^2)}{720}\,v^6+\ldots}
\end{array}
\end{equation}
\begin{remark}
The interest of these expansions is connected with the computation of these functions. By extension of the process followed by Bulirsch and Fukushima computing Jacobi elliptic functions (see Appendix). Nevertheless, there is still work to be done comparing that scheme with the possible advantages of using regularization.
\end{remark}
\subsection{Particular cases}
\label{sec:particularcases}
\begin{description}
\item $\bullet$ $n=0$. In this case, due to the choice of the initial conditions, we have ${\rm fng}(v)\equiv 1$. Moreover we have ${\rm sng}(v;0,m)= {\rm sn}(v,m)$, etc, {\it i.e.} the Jacobi elliptic functions with  elliptic modulus $m$.
\item $\bullet$ $m=0$. Here, based on the initial conditions, we have ${\rm dng}(v)\equiv 1$. Moreover ${\rm sng}(v;n,0)= {\rm sn}(v,n)$, etc, {\it i.e.} the Jacobi elliptic functions have an  elliptic modulus $n$ (que es negativo, thus we still needs to make a transformation; see (\ref{valores}) leading to $m_1$).

\item $\bullet$ $m=1$. In this case the differential equation is

\begin{equation}
 \frac{{\rm d}x}{{\rm d}v} = (1-x^2)\sqrt{1-n\,x^2}.
\end{equation}
For this quadrature we obtain
{\small
\begin{equation}
v=\frac{1}{2\sqrt{1-n}}{\rm ln}\frac{(1+x)}{(1-x)}\frac{(1-n x + \sqrt{(1-n)(1-n x^2)})}{(1+ n x + \sqrt{(1-n)(1-n x^2)})},
\end{equation}}
whose inversion is possible, because it is injective. 
\item $\bullet$ $m=n$. Now the differential equation is 
\begin{equation} \frac{{\rm d}x}{{\rm d}v} = (1-m x^2)\sqrt{1-x^2}.
\end{equation}
We obtain 
\begin{equation}
v=\frac{1}{\sqrt{1-m}} {\rm ArcTan}\Big(\sqrt{1-m}\frac{x}{\sqrt{1-x^2}}\Big)
\end{equation}
Again the inversion is possible because it is injective 
\begin{equation}
\tan(\sqrt{1-m}\,v) =\sqrt{1-m}\frac{x}{\sqrt{1-x^2}}
\end{equation}
More precisely, we have
\begin{equation}
x=\frac{\tan(\sqrt{1-m}\,w)}{\sqrt{1-m+\tan^2(\sqrt{1-m}\,v)}}
\end{equation}

Graphical examples for $n=m$ can be seen in Figs.~\ref{fig:Mahlernm05} and \ref{fig:Mahlernm095}.
\begin{figure}[h!]
\centering
{\label{fig:SeccionesEenergia}\includegraphics[scale=1.3]{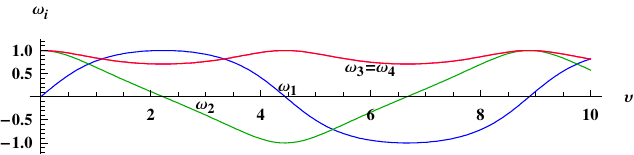}}
\vspace{-0.4cm}
\caption{\small Mahler $m=n=0.5$. Falta otra con valor m\'as extremo de $n=m$}
\label{fig:Mahlernm05}
\end{figure}
\begin{figure}[h!]
\centering
{\label{fig:SeccionesEenergia}\includegraphics[scale=1.3]{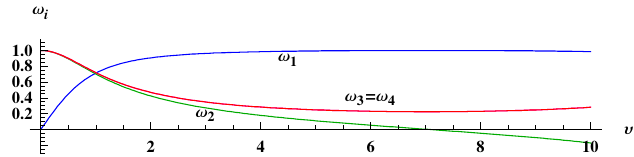}}
\vspace{-0.4cm}
\caption{\small Mahler $m=n=0.95$. Falta otra con valor m\'as extremo de $n=m$}
\label{fig:Mahlernm095}
\end{figure}
\item $\bullet$ $m=n=1$. In this case
\begin{equation}
v=\frac{x}{\sqrt{1-x^2}} 
\end{equation}
and finally, after inversion, it results
\begin{equation}
x=\frac{v}{\sqrt{1+v^2}}
\end{equation}

\item $\bullet$ $m=n=0$. In this case we recover the circular functions.

\item There are other particular cases related to unbounded trajectories, like the straight-lines which are expressed by elementary functions. This requires the signs of the coefficients to be the same, something that we have excluded when choosing our system.
\end{description}

\section{Addition formulas}
\label{sec:additionformulas}
In order to alleviate the notation, we introduce the following convention 
\begin{equation*}
\begin{array}{l}
\rm{sng}(a\,x;m,n)=\rm{s}_{ax},\quad\rm{cng}(a\,x;m,n)=\rm{c}_{ax},\\
\rm{dng}(a\,x;m,n)=\rm{d}_{ax},\quad\rm{fng}(a\,x;m,n)=\rm{f}_{ax}.
\end{array}
\end{equation*}
\begin{theorem}[Addition-Subtraction formulae for the 4-Mahler functions]
The addition and subtraction formulae for the 4-Mahler functions are given next. 
\begin{eqnarray}
&&\rm{sng}(x\pm y;m,n)=\label{eq:AdditionSubtraction4Mahler}\\
&&\dfrac{A\,(\rm{s}_{ax}\rm{c}_{ay}\rm{d}_{ay}\rm{f}_{ax}\pm \rm{s}_{ay}\rm{c}_{ax}\rm{d}_{ax}\rm{f}_{ay})}{\sqrt{(\rm{f}^2_{ax}\rm{f}^2_{ay}-m_1\rm{s}^2_{ax}\rm{s}^2_{ay})^2-n_1(\rm{s}_{ax}\rm{c}_{ay}\rm{d}_{ay}\rm{f}_{ax}\pm \rm{s}_{ay}\rm{c}_{ax}\rm{d}_{ax}\rm{f}_{ay})^2}}\nonumber\\[2ex]
&&\rm{cng}(x\pm y;m,n)=\nonumber\\
&&\dfrac{\rm{c}_{ax}\rm{c}_{ay}\rm{f}_{ax}\rm{f}_{ay}\mp \rm{s}_{ax}\rm{s}_{ay}\rm{d}_{ax}\rm{d}_{ay}}{\sqrt{(\rm{f}^2_{ax}\rm{f}^2_{ay}-m_1\rm{s}^2_{ax}\rm{s}^2_{ay})^2-n_1(\rm{s}_{ax}\rm{c}_{ay}\rm{d}_{ay}\rm{f}_{ax}\pm \rm{s}_{ay}\rm{c}_{ax}\rm{d}_{ax}\rm{f}_{ay})^2}}\nonumber\\[2ex]
&&\rm{dng}(x\pm y;m,n)=\nonumber\\
&&\dfrac{\rm{d}_{ax}\rm{d}_{ay}\rm{f}_{ax}\rm{f}_{ay}\mp \rm{s}_{ax}\rm{s}_{ay}\rm{c}_{ax}\rm{c}_{ay}}{\sqrt{(\rm{f}^2_{ax}\rm{f}^2_{ay}-m_1\rm{s}^2_{ax}\rm{s}^2_{ay})^2-n_1(\rm{s}_{ax}\rm{c}_{ay}\rm{d}_{ay}\rm{f}_{ax}\pm \rm{s}_{ay}\rm{c}_{ax}\rm{d}_{ax}\rm{f}_{ay})^2}}\nonumber\\[2ex]
&&\rm{fng}(x\pm y;m,n)=\nonumber\\
&&\dfrac{\rm{f}^2_{ax}\rm{f}^2_{ay}-m_1\rm{s}^2_{ax}\rm{s}^2_{ay}}{\sqrt{(\rm{f}^2_{ax}\rm{f}^2_{ay}-m_1\rm{s}^2_{ax}\rm{s}^2_{ay})^2-n_1(\rm{s}_{ax}\rm{c}_{ay}\rm{d}_{ay}\rm{f}_{ax}\pm \rm{s}_{ay}\rm{c}_{ax}\rm{d}_{ax}\rm{f}_{ay})^2}}\nonumber
\end{eqnarray}
where $A$, $a$, $m_1$ and $n_1$ are given in formula (43) (en la proposici—n 5).
\end{theorem}
\begin{proof}
Let us prove the formula corresponding to $\rm{sng}(x\pm y;m,n)$, the remaining ones are analogous. By Proposition~5 we have that
$$\rm{sng}(x\pm y;m,n)=A\dfrac{\rm{sn}(ax+ay\,;m_1)}{\sqrt{1-n_1\rm{sn}^2(ax+ay\,;m_1)}}.$$
Thus, using the addition and subtraction formulae for the Jacobi elliptic sine (see Appendix B) and assuming the following convention 
$$\rm{sn}(a\,x;m_1)=\rm{s}_{x},\quad\rm{cn}(a\,x;m_1)=\rm{c}_{x},\quad\rm{dn}(a\,x;m_1)=\rm{d}_{x},$$
we obtain
\begin{equation*}
\begin{array}{l}
\rm{sng}(x\pm y;m,n)=\\
\hspace{1cm}A\dfrac{\dfrac{\rm{s}_{x}\rm{c}_{y}\rm{d}_{y}\pm\rm{s}_{y}\rm{c}_{x}\rm{d}_{x}}{1-m_1\rm{s}_{x}^2\rm{s}^2_y}}{\sqrt{\dfrac{(1-m_1\rm{s}_{x}^2\rm{s}_{y}^2)^2-n_1(\rm{s}_x\rm{c}_y\rm{d}_y\pm\rm{s}_y\rm{c}_x\rm{d}_x)^2}{(1-m_1\rm{s}_{x}^2\rm{s}^2_y)^2}}},
\end{array}
\end{equation*}
simplifying denominators
\begin{equation}
\label{eq:AuxDoubleSine}
\begin{array}{l}
\rm{sng}(x\pm y;m,n)=\\
\hspace{0.8cm}A\dfrac{\rm{s}_{x}\rm{c}_{y}\rm{d}_{y}\pm\rm{s}_{y}\rm{c}_{x}\rm{d}_{x}}{\sqrt{(1-m_1\rm{s}_{x}^2\rm{s}_{y}^2)^2-n_1(\rm{s}_x\rm{c}_y\rm{d}_y\pm\rm{s}_y\rm{c}_x\rm{d}_x)^2}}.
\end{array}
\end{equation}
Finally, recalling that
\begin{eqnarray*}
&&\rm{s}_{x}=\frac{1}{A}\dfrac{\rm{sng}(ax;m,n)}{\rm{fng}(ax;m,n)}\\
&&\rm{c}_{x}=\dfrac{\rm{cng}(ax;m,n)}{\rm{fng}(ax;m,n)}\\
&&\rm{d}_{x}=\dfrac{\rm{sng}(ax;m,n)}{\rm{fng}(ax;m,n)},
\end{eqnarray*}
and likewise for $\rm{s}_{y}, \rm{c}_{y},\rm{d}_{y}$, if we multiply  numerator and  denominator in (\ref{eq:AuxDoubleSine}) by $\rm{fng}^2(ax;m,n)$ and $\rm{fng}^2(ay;m,n)$ we obtain (\ref{eq:AdditionSubtraction4Mahler}) after algebraic simplifications.
\end{proof}

\begin{corollary}
The formulae for the double angle of the 4-Mahler functions are given by
\begin{gather}
\begin{aligned}
\label{eq:Double4Mahler}
\rm{sng}(2x;m,n)&=&\dfrac{2A\,\rm{s}_{ax}\rm{c}_{ax}\rm{d}_{ax}\rm{f}_{ax}}{\sqrt{(\rm{f}_{ax}^4-m_1\,\rm{s}_{ax}^4)^2-n_1\,\big(2\,\rm{s}_{ax}\rm{c}_{ax}\rm{d}_{ax}\rm{f}_{ax}\big)^2}}\\
\rm{cng}(2x;m,n)&=&\dfrac{\rm{c}^2_{ax}\rm{f}^2_{ax}\mp \rm{s}^2_{ax}\rm{d}^2_{ax}}{\sqrt{(\rm{f}_{ax}^4-m_1\,\rm{s}_{ax}^4)^2-n_1\,\big(2\,\rm{s}_{ax}\rm{c}_{ax}\rm{d}_{ax}\rm{f}_{ax}\big)^2}}\\
\rm{dng}(2x;m,n)&=&\dfrac{\rm{d}_{ax}^2\rm{f}_{ax}^2\mp \rm{s}_{ax}^2\rm{c}_{ax}^2}{\sqrt{(\rm{f}_{ax}^4-m_1\,\rm{s}_{ax}^4)^2-n_1\,\big(2\,\rm{s}_{ax}\rm{c}_{ax}\rm{d}_{ax}\rm{f}_{ax}\big)^2}}\\
\rm{fng}(2x;m,n)&=&\dfrac{\rm{f}^4_{ax}-m_1\rm{s}^4_{ax}}{\sqrt{(\rm{f}_{ax}^4-m_1\,\rm{s}_{ax}^4)^2-n_1\,\big(2\,\rm{s}_{ax}\rm{c}_{ax}\rm{d}_{ax}\rm{f}_{ax}\big)^2}}
\end{aligned}
\end{gather}
\end{corollary}

\begin{corollary}
The formulae for the half angle of the 4-Mahler functions are given by
\begin{gather}
\begin{aligned}
\label{eq:Half4Mahler}
\rm{sng}(\frac{x}{2};m,n)=\;&A\,\sqrt{\dfrac{\rm{f}_{ax}-\rm{c}_{ax}}{\rm{f}_{ax}+\rm{d}_{ax}-n_1(\rm{f}_{ax}-\rm{c}_{ax})}}\\
\rm{cng}(\frac{x}{2};m,n)=\;&\sqrt{\dfrac{\rm{d}_{ax}+\rm{c}_{ax}}{\rm{f}_{ax}+\rm{d}_{ax}-n_1(\rm{f}_{ax}-\rm{c}_{ax})}}\\
\rm{dng}(\frac{x}{2};m,n)=\;&\sqrt{\dfrac{(\rm{c}_{ax}+\rm{d}_{ax})(\rm{f}_{ax}+\rm{d}_{ax})}{(\rm{f}_{ax}+\rm{c}_{ax})(\rm{f}_{ax}+\rm{d}_{ax})-n_1(\rm{f}_{ax}^2-\rm{c}_{ax}^2)}}\\
\rm{fng}(\frac{x}{2};m,n)=\;&\sqrt{\dfrac{\rm{f}_{ax}+\rm{d}_{ax}}{\rm{f}_{ax}+\rm{d}_{ax}-n_1(\rm{f}_{ax}-\rm{c}_{ax})}}
\end{aligned}
\end{gather}
\end{corollary}
%
\subsection{On the numerical computation of $\omega_i$ functions by extending Bulirsch-Fukushima method}
\label{sec:computing}
As we know Jacobi elliptic functions are defined by some ratios of $\theta_i$ Jacobi functions. This way of handling the Jacobi elliptic functions is  convenient due to the fast convergency of those series. Nevertheless, at present, fast numeric codes compete with this classical analytic approach. More precisely, in order to implement those codes {\sl addition formulas} compute Jacobi elliptic functions are basic expressions in that process (see Fukushima\cite{Fukushima2013,Fukushima2014}). 

 We can extend those expressions to the $\omega_i$ functions.  Thus, as Fukushima explains, the algorithm is made of three steps:\\
(i) the forward transformation defined by  (Corollary 2, Half arguments formulas: (\ref{eq:Half4Mahler}) reducing the values of $\omega_i$ by a number of iterations;\\
(ii) evaluation of the Mac-Laurin series expansions given by (\ref{algoritmo2}) and;\\ (iii) the backward transformation  (Corollary 1: Double arguments formulas  (\ref{eq:Double4Mahler})) as many times  as the forward transformation.

Details of the implementation of this process will be given in \cite{Crespo2015}.
\section{On the case $N=5$} 
\label{sec:N5}
As we have pointed out in the Introduction,  hyperelliptic integrals appear in (\ref{eq:EESn}) when $N\geq 5$. Thus it is convenient to see in some detail the case  $N=5$, the lower system belonging to this category. 

 Thus, as before, we start keeping the notation used in lower dimension
 \begin{equation}\label{sistema5}
 \begin{array}{l}
 \displaystyle{\omega_1'= \alpha_1\,\omega_2\,\omega_3\,\omega_4\,\omega_5}, \\\displaystyle{ \omega_2'= \alpha_2\,\omega_1\,\omega_3\,\omega_4\,\omega_5},\\\displaystyle{\omega_3'= \alpha_3\,\omega_1\,\omega_2\omega_4\,\omega_5},\\\displaystyle{ \omega_4'= \alpha_4\,\omega_1\,\omega_2\,\omega_3\,\omega_5},\\\displaystyle{ \omega_5'= \alpha_5\,\omega_1\,\omega_2\,\omega_3\,\omega_4},
 \end{array}
 \end{equation}
with given initial conditions $\omega(0)$. As examples in Figs. \ref{fig:MahlerP02N04M07}
and \ref{fig:MahlerPmenos2Nmenos1M04} we present two set of functions of the 
5-EES family.
\begin{figure}[h!]
\centering
{\label{fig:}\includegraphics[scale=1.3]{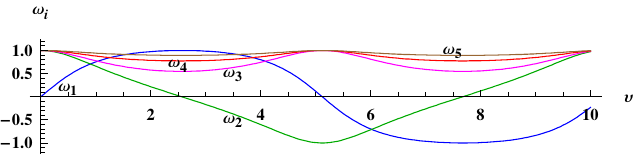}}
\vspace{-0.4cm}
\caption{\small 5-Mahler system graphs for $p=0.2, n=0.4, m=0.7$.}
\label{fig:MahlerP02N04M07}
\end{figure}
\begin{figure}[h!]
\centering
{\label{fig:}\includegraphics[scale=1.3]{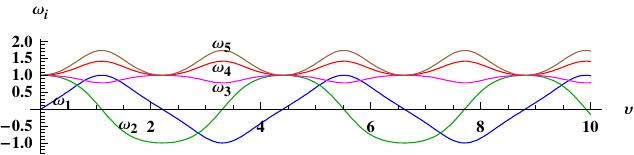}}
\vspace{-0.4cm}
\caption{\small 5-Mahler system graphs  for  $p=-2, n=-1, m=0.4$.}
\label{fig:MahlerPmenos2Nmenos1M04}
\end{figure}

We will proceed as in the lower dimensions $N=3,4$, considering alternative procedures to the classic solution based on the direct reduction to hyperelliptic integrals. In other words:\medskip\par\noindent
 (i) we  introduce the functions $u_i^j(v)$, (where we maintain the notation) ratios of the $\omega_i$
\begin{equation}\label{ratios5}
u_i^j=\frac{\omega_i}{\omega_j}, \quad  i\neq j, \qquad  u_j^j=\frac{1}{\omega_j},
\end{equation}
in the domain of definition of $\omega_j$.\medskip\par\noindent
(ii) In the rest of the section we will study the effect of the introduction of some possible regularizations, namely two of them
\begin{itemize}
\item $\displaystyle{{\rm d}v^* = \omega_5\,{\rm d}v.}$
\item $\displaystyle{{\rm d}v^*=\omega_3\,\omega_4\,\omega_5\,{\rm d}v.}$
\end{itemize}
Again, we have to keep in mind that with the notation used in the above regularizations, the new variable $v^*$ is different from one case to the other.
\subsection{The  ${\rm d}v^*/{\rm d}v=\omega_5$ regularization.}
\label{sec:omega5}
Then, associated to the ratios $u_i$, if we carry out the regularization
\begin{equation}\label{regularization5}
{\rm d}v =u_5\,{\rm d}v^*.
\end{equation}
 we have the following regularized differential system
\begin{equation}\label{sistema43}
\begin{array}{l}
\displaystyle{\frac{{\rm d}u_1}{{\rm d}v^*}= C_1^5\,u_2u_3u_4}, \\[1.5ex]
\displaystyle{\frac{{\rm d}u_2}{{\rm d}v^*}= C_2^5\,u_1u_3u_4},\\[1.5ex]
\displaystyle{\frac{{\rm d}u_3}{{\rm d}v^*}= C_3^5\,u_1u_2u_4},\\[1.5ex]
\displaystyle{\frac{{\rm d}u_4}{{\rm d}v^*}= C_4^5\,u_1u_2u_3},
\end{array}
\end{equation}
with IC $u_i(0)=u_i^0=\omega_i^0/\omega_j^0$, $i=1,\ldots, 4$. 

Thus, dividing the integral  $\alpha_1\omega_5^2-\alpha_5\omega_1^2=C_1^5$ by $\omega_5^2$ we write: $u_5^2=(\alpha_1-\alpha_5\,u_1^2)/C_1^5$. Then, we obtain
\begin{equation}\label{regularizacion5}
v= \sqrt{\frac{\alpha_1}{C_1^5}}\int\sqrt{1-n_2\,[u_1(v^*)]^2}{\rm d}v^*,
\end{equation}
where $n_2=\alpha_5/\alpha_1$ and $u_1(v^*)$ is a function solution of the system (\ref{sistema43}); quadrature which will be solved numerically. As we know that the  solution of (\ref{sistema43}) can be obtained by undetermined coefficients, making use of the $4$-{\sl Mahler functions} defined by the system (\ref{sistemasng}), but in the variable $v^*$. In other words, the previous form of the solution represents an alternative to the use of hyperelliptic integrals for solving (\ref{eq:EESn}) for $N=5$. Or, in a more precise form, we have separated geometry from dynamics. The trajectory is expressed by Jacobi or Mahler functions, meanwhile the quadrature of the parametrization (\ref{regularizacion5}) will lead generically to a  hyperellictic integral.
\subsection{The  ${\rm d}v^*/{\rm d}v=\omega_3\omega_4\omega_5$ regularization.}
\label{sec:omega345}
 Let us consider again the system 5-EES (\ref{sistema5}). Now we try the regularization
\begin{equation}\label{regularization345}
\frac{{\rm d}v^*}{{\rm d}v}=\omega_3\,\omega_4\,\omega_5.
\end{equation}
in a domain where $\omega_3\,\omega_4\,\omega_5\neq 0$. This means that the system reduces to 
\begin{equation}\label{generic22}
\frac{{\rm d}\omega_1}{{\rm d}v^*}= \alpha_1\,\omega_2, \quad
\frac{{\rm d}\omega_2}{{\rm d}v^*}= \alpha_2\,\omega_1,
\end{equation}
and three quadratures associated to $\omega_3$, $\omega_4$ and $\omega_5$. In fact, they are not needed because the integrals allow to write
$\omega_i^2=C_1^i-\alpha_i\omega_1^2$, $(i=3,4,5)$. Remember that  $C_1^i$ are constants, functions of the initial conditions.

Assuming the bounded case we can always choose, by scaling and transformation of functions, that  $\alpha_1=1,\, \alpha_2=-1$. In other words we have 
\begin{equation}
\omega_1(v^*)=\sin v^*, \qquad \omega_2(v^*)=\cos v^*,
\end{equation}
Then, the quadrature (\ref{regularization345}), taking into account the previous mentioned integrals, we have 
\begin{equation}\label{cuadratura3}
\lambda v = \int\frac{{\rm d}v^*}{\sqrt{\prod_{i=3}^5 (1-\beta_i\sin^2 v^*)}}
\end{equation}
where $\beta_i$ and $\lambda$ are functions of $C_1^i$ and $\alpha_i$.  This lead us, in the generic case, to a hyperelliptic quadrature.

\paragraph{Dealing with the $5$-Mahler System.} In what follows we choose as the basic system in $N=5$ a Mahler type system
\begin{equation}\label{Mahler5}
\begin{array}{l}
\displaystyle{\omega_1'= \omega_2\,\omega_3\,\omega_4\,\omega_5}, \\
\displaystyle{ \omega_2'= -\omega_1\,\omega_3\,\omega_4\,\omega_5},\\
\displaystyle{\omega_3'= -m\,\,\omega_1\omega_2\,\omega_4\,\omega_5},\\
\displaystyle{ \omega_4'= -n\,\omega_1\,\omega_2\,\omega_3\,\omega_5},\\
\displaystyle{ \omega_5'= -p\,\omega_1\,\omega_2\,\omega_3\,\omega_4},
\end{array}
\end{equation}
with initial conditions $(0,1,1,1,1)$. 

Moreover, apart from adjusting coefficients, an alternative form of dealing with (\ref{cuadratura3}) is to make a change of variable  $\sin v^* =x$. Then, the corresponding new expression for the regularization 
 is given by
\begin{equation}\label{cuadratura33}
\lambda\,{\rm d}v = \frac{{\rm d}x}{\sqrt{(1-x^2) (1-m\,x^2) (1-n\,x^2) (1-p\,x^2)}}.
\end{equation}
Denoting 
$$w=\lambda\,v$$ 
we define by {\rm Amg} (generalized amplitude) the inverse function
\begin{equation}\label{amplitudgeneralizada}
 v^*= {\rm Amg}(w;p,m,n).
\end{equation}
Then, by analogy with the notation introduced in lower dimensions, we propose to  write
\begin{equation}
\sin\,  v^*= \sin\,{\rm Amg}(w;p,n,m)\equiv {\rm Sng}(w;p,n,m)
\end{equation}
In other words, we define ${\rm Sng}$
\begin{equation}\label{sng}
x=x(w;p,n,m)={\rm Sng}(w;p,n,m)
\end{equation}
as the three-parameter function solution of the differential equation
\begin{equation}\label {ed1}
 \Big(\frac{{\rm d}x}{{\rm d}w}\Big)^2=(1-x^2)(1-p\,x^2)(1-n\,x^2)(1-m\,x^2).
\end{equation}
In the rest of this paper we restrict ourselves to the domain of parameters $\Delta=\{(p,n,m)\in [0,1]\times[0,1]\times[0,1]\}$.
\par
Then, associated with ${\rm Sng}$ we introduce the following functions
\begin{equation}\label{CngDngFngHng}
\begin{array}{l}
\displaystyle{{\rm Cng}(w;p,n,m)=\pm\sqrt{1-{\rm Sng}^2(w;p,n,m)},}\\[1.2ex]
\displaystyle{{\rm Dng}(w;p,n,m)=\sqrt{1-m\,{\rm Sng}^2(w;p,n,m)},}\\[1.2ex]
\displaystyle{{\rm Fng}(w;p,n,m)=\sqrt{1-n\,{\rm Sng}^2(w;p,n,m)},}\\[1.2ex]\,\displaystyle{{\rm Hng}(w;p,n,m)=\sqrt{1-p\,{\rm Sng}^2(w;p,n,m)}}.
\end{array}
\end{equation}
To simplify the notation, we will write in some expressions
\begin{eqnarray*}
&&{\rm Sng}(w;p,n,m)\equiv {\rm Sng}, \quad {\rm Cng}(w;p,n,m)\equiv {\rm Cng}, \\ &&{\rm Dng}(w;p,n,m)\equiv {\rm Dng}, \quad {\rm Fng}(w;p,n,m)\equiv {\rm Fng},\\
&&{\rm Hng}(w;p,n,m)\equiv {\rm Hng}.
\end{eqnarray*}
Then, we write again  (\ref{Mahler5}) as the following IVP
\begin{equation}\label{Mahler51}
\begin{array}{l}
\displaystyle{\frac{{\rm d}\,{\rm Sng}}{{\rm d}w}= {\rm Cng}\,{\rm Dng}\,{\rm Fng}\,{\rm Hng}}, \\  [1.8ex]
\displaystyle{\frac{{\rm d}\,{\rm Cng}}{{\rm d}w}= -{\rm Sng}\,{\rm Dng}\,{\rm Fng}\,{\rm Hng}}\\ [1.8ex]
\displaystyle{\frac{{\rm d}\,{\rm Dng}}{{\rm d}w}= -m\,{\rm Sng}\,{\rm Cng}\,{\rm Fng}\,{\rm Hng}}\\ [1.8ex]
\displaystyle{\frac{{\rm d}\,{\rm Fng}}{{\rm d}w}= -n\,{\rm Sng}\,{\rm Cng}\,{\rm Dng}\,{\rm Hng}},\\[1.8ex]
\displaystyle{\frac{{\rm d}\,{\rm Hng}}{{\rm d}w}= -p\,{\rm Sng}\,{\rm Cng}\,{\rm Dng}\,{\rm Fng}}
\end{array}
\end{equation}
with initial conditions (0,1,1,1,1). Note that in agreement with (\ref{CngDngFngHng}), the integrals take the following form
\begin{equation}
\begin{array}{l}
{\rm Cng}^2+{\rm Sng}^2=1, \qquad {\rm Dng}^2+m\,{\rm Sng}^2=1, \\[1.2ex]
{\rm Fng}^2+n\,{\rm Sng}^2=1, \qquad {\rm Hng}^2+p\,{\rm Sng}^2=1.
\end{array}
\end{equation}
We are not going to deal with the generic study of our system (\ref{Mahler51}). It is out of the scope of this paper. In the last Section we will restrict to analyze some particular cases
\section{$N=5$: Some particular cases} 
\label{sec:appendixpp}
Like in previous dimensions, we consider two particular cases
\subsection{The case $p=0$.} 
Now, according to (\ref{CngDngFngHng}), we have ${\rm Hng}\equiv 1$. This corresponds to the previous studied case: 4-Mahler system.
\subsection{The case $p=n$.} 
As we have just pointed out, a particular case of (\ref{cuadratura3}) we will consider now two of the $\beta_i$ equals. According to the notation introduced, 
we write 
\begin{equation}\label{cuadratura&Pi}
\lambda\,v=\int_0^{\tilde v^*}\frac{d\vartheta}{(1-n\sin^2\vartheta)\sqrt{1-m\sin^2\vartheta}}, 
\end{equation}
\begin{remark}
In relation to the quadrature (\ref{cuadratura&Pi}) the reader will remember  that this is precisely the Legendre third elliptic integral\footnote{Dealing with the search of fast numerical algorithms for the computation of the third elliptic integral Fukushima \cite{Fukushima2013,Fukushima2014} singles out in a recent paper that by a number of transformations the domain of $n$ and $m$ are reduced as
$$ 0 < m <1, \qquad -\sqrt{m} <n < \frac{m}{1+\sqrt{1-m}}. $$
This fact has to be in mind (incluir gr‡fico de este dominio) in order to study the $\omega_i$ thinking on  applications to those integrals\ldots} $\Pi(\tilde v^*;m,n)$. Thus, for the particular cases $n=0$ and $n=m$, we encounter the other Legendre elliptic integrals:
\begin{eqnarray*}
&&F(\varphi,m)=\int_0^{\varphi}\frac{d\vartheta}{\sqrt{1-m\sin^2\vartheta}}= \Pi(\varphi,0,m)\\
&&E(\varphi,m)=\int_0^{\varphi}\sqrt{1-m\sin^2\vartheta}\,d\vartheta,\\
&&\hspace{1.2cm}=(1-m)\,\Pi(\varphi,m,m) + m \frac{\sin(2\varphi)}{2\sqrt{1-m\sin^2\varphi}}.
\end{eqnarray*}
\end{remark} 
Denoting
$$w=\lambda\,v,$$ 
we define as {\rm Amg} (generalized amplitude) the inverse function
\begin{equation}
\tilde v^*= {\rm Amg}(w;n,n,m).
\end{equation}
Then, by analogy with the notation introduced in lower dimensions, we propose to  write
\begin{equation}
\sin\, \tilde v^*= \sin\,{\rm Amg}(w;n,n,m)\equiv {\rm Sng}(w;n,m)
\end{equation}
For later use, we also include here the expression for our particular case of (\ref{cuadratura33})
\begin{equation}\label{ed1}
{\rm d}w = \frac{{\rm d}x}{(1-n\,x^2) \sqrt{(1-x^2) (1-m\,x^2) }}.
\end{equation}
From our initial conditions  we  have ${\rm Hng}\equiv {\rm Fng}$. Then, from   (\ref{Mahler5}) we immediately obtain that $\omega_4\equiv \omega_5$, and that these functions satisfy the following IVP
\begin{equation}\label{regularizeSng00}
\begin{array}{l}
\displaystyle{\frac{{\rm d}\,{\rm Sng}}{{\rm d}w}= {\rm Cng}\,{\rm Dng}\,{\rm Fng}^2}, \\  [1.8ex]
\displaystyle{\frac{{\rm d}\,{\rm Cng}}{{\rm d}w}= -{\rm Sng}\,{\rm Dng}\,{\rm Fng}^2}\\ [1.8ex]
\displaystyle{\frac{{\rm d}\,{\rm Dng}}{{\rm d}w}= -m\,{\rm Sng}\,{\rm Cng}\,{\rm Fng}^2}\\ [1.8ex]
\displaystyle{\frac{{\rm d}\,{\rm Fng}}{{\rm d}w}= -n\,{\rm Sng}\,{\rm Cng}\,{\rm Dng}\,{\rm Fng}},
\end{array}
\end{equation}
con las condiciones iniciales (0,1,1,1). 

Again by a regularization  $w\rightarrow \tilde v$ given by
\begin{equation}\label{cuadraturapi0}
\frac{{\rm d}\tilde v}{{\rm d}w} = {\rm Fng}
\end{equation}
transforms (\ref{regularizeSng00}) in a regularized system 
%
which is a 4-Mahler system in the new variable. 

After we have solved the regularized system, we still need to compute  the quadrature associated to the differential relation  (\ref{cuadraturapi0}). Explicitly we have
\begin{equation}
{\rm d}w=\int\frac{{\rm d}\tilde v}{\sqrt{1-n \,{\rm Sng}^2(w(\tilde v))}} 
\end{equation}
We will give details of this process, both from the analytical and numerical point of view, in a forthcoming paper.
\section{On the application to the free rigid body}
\label{sec:FRB}
We will apply what we have presented in previous sections to the description of the solution of the free rigid body. We will do that formulating the system in symplectic Andoyer variables.

\subsection{The solution in Andoyer variables}
\label{sec:classical}
Let us consider the Hamiltonian of the free rigid body expressed in Andoyer's variables $(\lambda,\mu,\nu,\Lambda,M,N)$ which takes the form
\begin{equation}
\mathcal{H}=\frac{1}{2}(a_1\sin^2\nu + a_2\cos^2\nu)(M^2-N^2) + \frac{a_3}{2}N^2,
\end{equation}
where $(a_1, a_2, a_3)=(1/A, 1/B,1/C)$ with $(A,B,C)$  the principal moments of inertia. Note that in applications we will study the influence of  $B$, which will be taken as physical parameter $A\leq B\leq C$, join with $C<A+B$. 
The differential system is given by three equations
\begin{eqnarray}
&&\frac{\mathrm{d}\nu}{\mathrm{d}t}=\phantom{-}\frac{\partial\mathcal{H}}{\partial N}=N(a_3-a_1\sin^2\nu-a_2\cos^2\nu),\label{nuPunto}\\[0.7ex]
&&\frac{\mathrm{d}N}{\mathrm{d}t}=-\frac{\partial\mathcal{H}}{\partial\nu}=(a_2-a_1)(M^2-N^2)\sin\nu\cos\nu,\label{NPunto}\\
&&\frac{\mathrm{d}\mu}{\mathrm{d}t}=\phantom{-}\frac{\partial\mathcal{H}}{\partial M}=M(a_1\sin^2\nu+a_2\cos^2\nu),\label{muPunto}
\end{eqnarray}
and the other three $(\lambda,\Lambda,M)$ which are integrals. Usually  we integrate first  the system defined by $N$ and $\nu$. More precisely, we solve the Euler system, associated with those variables. Then, the functions solution $N(t)$ and $\nu(t)$ are given making use of the Jacobi elliptic functions
\begin{equation}\label{solutionnu}
\begin{array}{l}
\displaystyle{\sin\nu(t) = \frac{{\rm cn}(s\,t;\,m)}{\sqrt{1+n^*{\rm sn}^2(s\,t;\,m)}}}, \\[2.5ex]
\displaystyle{\cos\nu(t) = \sqrt{1+n^*}\frac{{\rm sn}(s\,t;\,m)}{\sqrt{1+n^*{\rm sn}^2(s\,t;\,m)}}}, \\[2.5ex]
\displaystyle{\hspace{0.4cm}N(t)= R\, {\rm dn}(s\,t;\,m),}
\end{array}
\end{equation}
where
\begin{eqnarray*}
&&R^2 = M^2\frac{C(1-dA)}{C-A}, \quad n^*=-n=\frac{C(B-A)}{A(C-B)},\\
&&m= \frac{(B-A)(dC-1)}{(C-B)(1-dA)}, \quad s^2 = M^2\frac{(C-B)(1-dA)}{ABC},
\end{eqnarray*}
with $d=2h/M^2$. 

\begin{remark} 
The reader will notice that $\sin\nu(t)$  and  $\cos\nu(t)$ are Mahler functions. Moreover from (\ref{JacobiMahler}) we know that $N(t)$ is a ratio of Mahler functions 
\end{remark}

Finally, we obtain $\mu(t)$ by means of a quadrature: 
\begin{equation}\label{cuadraturamu}
\mu=M\int (a_1\sin^2\nu(t)+a_2\cos^2\nu(t))\,\mathrm{d}t
\end{equation}
which is finally expressed by means of a linear function of time and the Legendre third elliptic integral. Integral whose solution Jacobi gave making use of his elliptic and related functions. 

\subsection{On alternative approaches}

We proceed here in a different form than the previous Section \ref{sec:classical}. Making use of the Hamiltonian function, we may separate variables in the system defined by (\ref{nuPunto})-(\ref{NPunto}). More precisely, we denote  
\begin{equation}
 n_1=\frac{a_1-a_2}{d-a_2}, \quad m_1= \frac{a_1-a_2}{a_3-a_2}
\end{equation}
and $\Omega= (d-a_2)(a_3-a_2)$, where we assume $a_3\neq a_2$ and $d\neq a_2$; (the case of equality has to be treated separately). Then, the equation (\ref{nuPunto}) may be written in the form
\begin{equation}
M\Omega\,\mathrm{d}t=\frac{\mathrm{d}\nu}{\sqrt{(1-n_1\sin^2\nu)(1-m_1\sin^2\nu)}}
\end{equation}
where  $n_1=n_1(d,a_2)$ and  $m_1=m_1(d,a_2)$, that is to say, we may study the system under the influence of the intermediate moment of inertia and the value of the Hamiltonian, keeping fixed the other parameters. Again, we have to distinguish circulation and libration patterns, but we do not need to go into details of that procedure here.
\medskip\par\noindent
$\bullet$ In a more detailed form, we see that from (\ref{sistema4mnsolution}) and (\ref{solutionnu}) we may write 
\begin{eqnarray*}
&&\sin\nu(t) = A_1\,{\rm cng}(w;n,m), \\ 
&&\cos\nu(t) = A_2\, {\rm sng}(w;n,m), \\
&&N=A_3\,{\rm dng}(w;n,m)/{\rm fng}(w;n,m),
\end{eqnarray*}
where $A_i$ are quantities depending on the previous constants.
\par
 The quadrature (\ref{cuadraturamu}) of the Andoyer angle variable $\mu$, now takes the form:
\begin{eqnarray}
&&\mu = M\int (a_1\sin^2\nu+a_2\cos^2\nu)\,\mathrm{d}t\nonumber \\
&&\hspace{0.3cm}= M\int (\tilde a_1 {\rm sng}^2t + \tilde a_2 {\rm cng}^2t) \,\mathrm{d}t\\
&&\hspace{0.3cm}= M\tilde a_2\,t + a_1^* \int {\rm sng}^2t \,\mathrm{d}t\nonumber\end{eqnarray}
\par\noindent
$\bullet$ Finally, from what we have seen in Sect. \ref{sec:appendixpp} we find that 
$$ \sin\mu = {\rm Sng}(w;n,m), \qquad \cos\mu = {\rm Cng}(w;n,m)$$
In other words, depending on the use of ${\rm sng}$, etc. or ${\rm Sng}$, etc. we reach the third Legendre elliptic integral in two different forms. Comparisons of the pros and cons of their use, versus the classic approach based on ${\rm sn}$, etc. Jacobi functions, is in progress.
\section{Appendices}
\label{sec:Appendices}
{\bf Appendix A: On the ratios of Jacobi $\theta_i$ functions as solutions of 3-EES.}

From Lawden \cite{Lawden} (Chp.~1) we borrow the following 3-EES differential systems satisfied by the ratios of the Jacobi $\theta_i$ functions
\begin{eqnarray}
&&\label{ratio41}\frac{{\rm d}\phantom{-}}{{\rm d}v}\Big(\frac{\theta_1}{\theta_4}\Big)=\theta_4^2(0)\frac{\theta_2}{\theta_4}\frac{\theta_3}{\theta_4},\\
&&\frac{{\rm d}\phantom{-}}{{\rm d}v}\Big(\frac{\theta_2}{\theta_4}\Big)=-\theta_3^2(0)\frac{\theta_1}{\theta_4}\frac{\theta_3}{\theta_4},\\
&&\frac{{\rm d}\phantom{-}}{{\rm d}v}\Big(\frac{\theta_3}{\theta_4}\Big)=-\theta_2^2(0)\frac{\theta_1}{\theta_4}\frac{\theta_2}{\theta_4},
\end{eqnarray}
etc.  We find convenient to introduce the notation $x_{ij}=\theta_j/\theta_i$ and the reparametrization $v\rightarrow \tau$ given by ${\rm d}\tau=\sqrt{2{\rm K}/\pi}\,{\rm d}v$, with $x_{ij}'={\rm d}x_{ij}/{\rm d}\tau$. Thus, taking into account the values of $\theta_i(0)$, where 
$k^2=m$, $k^2+{k'}^2=1$ and ${\rm K}(m)$ is the the complete Legendre first elliptic integral, we write those IVP  systems as follows. Note that, as was pointed out in Crespo and Ferrer \cite{Crespo2015}, considering the sign of the coefficients, we may distinguish
\medskip\par\noindent
$\bullet$ Two bounded systems:
\begin{eqnarray*}
&&x_{41}'=k'\,x_{42}\,x_{43},\\
&&x_{42}'=-\,x_{41}\,x_{43},\\
&&x_{43}'=-k\,x_{41}\,x_{42}, \qquad (0,\sqrt{k/k'},1/\sqrt{k'})
\end{eqnarray*}
and
\begin{eqnarray*}
&&x_{31}'=x_{32}\,x_{34},\\
&&x_{32}'=-k'\,x_{31}\,x_{34},,\\
&&x_{34}'=k\,x_{31}\,x_{32}, \qquad (0,\sqrt{k},\sqrt{k'})
\end{eqnarray*}
\par\noindent
$\bullet$ Two unbounded systems:
\begin{eqnarray*}
&&x_{21}'=k\,x_{23}\,x_{24},\\
&&x_{23}'=k'\,x_{21}\,x_{24},,\\
&&x_{24}'=x_{21}\,x_{23}, \qquad (0,1/\sqrt{k},\sqrt{k'/k})
\end{eqnarray*}
and 
\begin{eqnarray*}
&&x_{12}'=-k\,x_{13}\,x_{14},\\
&&x_{13}'=-x_{12}\,x_{14},,\\
&&x_{14}'=-k'\,x_{12}\,x_{13}, \quad (1,\sqrt{(k'+1)/k},\sqrt{(k'+1)/k}).
\end{eqnarray*}
Then, we may express those ratios as functions the Jacobi elliptic functions and their Glashier ratios.



\medskip\par\noindent
{\bf Appendix B: Transformations and addition formulas for Jacobi elliptic functions.}

For the benefit of the reader we bring here some well known transformations involving the {\sl elliptic modulus}. They may be found in any handbook of elliptic functions (remember that, depending on the authors, two notations are used: `modulus' or `parameter' related by $k^2\equiv m$, and their complementaries). Those formulas should be used for the reduction to the normal case of some of the particular cases mentioned along the paper.
\medskip\par\noindent
$\bullet$ {\it Negative parameter} \\
Let $m$ be a positive number and write 
\begin{equation}\label{eq:changemnegative}
\mu=\frac{m}{1+m}, \qquad \mu_1=\frac{1}{1+m}, \qquad v=\frac{u}{\sqrt{\mu_1}}.
\end{equation}
Then, 
\begin{eqnarray*}
&&{\rm sn}(u\,;-m)=\sqrt{\mu_1}\,\frac{{\rm sn}(v\,;\,\mu)}{{\rm dn}(v\,;\,\mu)},\\
&&{\rm cn}(u\,;-m)=\frac{{\rm cn}(v\,;\,\mu)}{{\rm dn}(v\,;\,\mu)},\\
&&{\rm dn}(u\,;-m)=\frac{1}{{\rm dn}(v\,;\,\mu)}.
\end{eqnarray*}
Thus elliptic functions with negative parameter may be expressed by elliptic functions with a positive parameter. Note that $0<\mu<1$.\par

A final comment related to the complete elliptic integral of first kind is due here. Unlike {\sl Maple}, the software {\sl Mathematica} yields the following result
\begin{equation}
\int_0^{\pi/2}\frac{\mathrm{d}\phi}{\sqrt{1-m\sin^2\phi}}=\frac{1}{\sqrt{1-m}}\,\mathrm{K}\left(\frac{m}{m-1}\right)
\end{equation}
for $\forall m\leq 1$, instead of the expected result $\mathrm{K}(m)$. By applying the previous change (\ref{eq:changemnegative}), we have that, being $m$ a positive number,
\begin{equation}
\mathrm{K}(-m)=\frac{1}{\sqrt{1+m}}\,\mathrm{K}\left(\frac{m}{1+m}\right)=\sqrt{\mu_1}\,\mathrm{K}(\mu)
\end{equation}
which is exactly the same result given by {\sl Mathematica} for $m<0$.

\medskip\par\noindent
$\bullet$ {\it Reciprocal parameter} \\
Denoting now $v=\sqrt{m}u$, we have
\begin{eqnarray*}
&&{\rm sn}(u\,;m)=\frac{1}{\sqrt{m}}\,{\rm sn}(v\,;m^{-1}),\\
&&{\rm cn}(u\,;m)={\rm dn}(v\,;m^{-1}),\\
&&{\rm dn}(u\,;m)={\rm cn}(v\,;m^{-1}).
\end{eqnarray*}
This is Jacobi's {\it real transformation}. If $m>1$, then $m^{-1}<1$, thus elliptic functions whose parameter is greater than $1$ are related to the ones whose parameter is less than $1$. In short there is no loss of generality assuming $0\leq m\leq 1$.
\medskip\par\noindent
$\bullet$ {\it Decrease of parameter} \\
\begin{equation}
\mu=\Big(\frac{1-\sqrt{m_1}}{1+\sqrt{m_1}}\Big)^2,  \qquad v=\frac{u}{1+\sqrt{\mu}}.
\end{equation}
\begin{eqnarray*}
&&{\rm sn}(u\,;m)=\frac{(1+\sqrt{\mu}){\rm sn}(v\,;\,\mu)}{1+\sqrt{\mu}\,{\rm sn}^2(v\,;\,\mu)},\\[1ex]
&&{\rm cn}(u\,;m)=\frac{{\rm cn}(v\,;\,\mu)\,{\rm dn}(v\,;\,\mu)}{1+\sqrt{\mu}\,{\rm sn}^2(v\,;\,\mu)},\\[1ex]
&&{\rm dn}(u\,;m)=\frac{1-\sqrt{\mu}\,{\rm sn}^2(v\,;\,\mu)}{1+\sqrt{\mu}\,{\rm sn}^2(v\,;\,\mu)}.
\end{eqnarray*}
This is Gauss transformation or the {\it descending} Landen transformation, which makes elliptic functions to depend on functions with a smaller parameter.

Note that, making use of the double angle, we may also write
\begin{equation}
{\rm dn}(u\,;m)=\frac{\sqrt{\mu}\,{\rm cn}(2v\,;\,\mu) + {\rm dn}(2v\,;\,\mu)}{1+\sqrt{\mu}}.
\end{equation}
There are analogous expressions for the increase of parameter. For a recent study where generalized formules are given, see \cite{Khare}.
 \medskip\par\noindent
$\bullet$ {\it Addition formulae} \\
Complementing previous transformations, we collect also here the {\sl addition formulae} 
\begin{eqnarray*}
&&{\rm sn}(\alpha+\beta) =\frac{{\rm sn}\,\alpha \, {\rm cn}\,\beta\,{\rm dn}\,\beta + 
{\rm sn}\,\beta \, {\rm cn}\,\alpha\, {\rm dn}\,\alpha}{1-m\,{\rm sn}^2\alpha \, {\rm sn}^2\beta},\\
&&{\rm cn}(\alpha+\beta) =\frac{{\rm cn}\,\alpha \, {\rm cn}\,\beta - {\rm sn}\,\alpha\,  
 {\rm sn}\,\beta\,{\rm dn}\alpha\, {\rm dn}\,\beta}{1-m\,{\rm sn}^2\alpha \, {\rm sn}^2\beta},\\
&&{\rm dn}(\alpha+\beta) =\frac{{\rm dn}\,\alpha \, {\rm dn}\beta - m\, {\rm sn}\,\alpha\, {\rm sn}\,\beta\,{\rm cn}\,\alpha  \,{\rm cn}\,\beta}{1-m\,{\rm sn}^2\alpha \, {\rm sn}^2\beta},
\end{eqnarray*}
which we have generalized for the new functions; more precisely this has been done for the 4-EES Mahler system. 
\section*{Acknowledgements}
Support from Research Agencies of Spain is acknowledged. They came in the form of research projects MTM 2012-31883, of the Ministry of Science, and  12006\-/PI/09 from Fundaci\'on S\'eneca of the Autonomous Region of Murcia.

\end{document}